\documentclass[a4paper,12pt]{article}

\usepackage{amsmath}
\usepackage{amsthm}
\usepackage{amssymb}
\usepackage{cite}
\usepackage{url}
\usepackage{footmisc}
\usepackage{graphicx}
\usepackage{epstopdf}
\usepackage{chngcntr}
\usepackage{enumitem}
\usepackage{hhline}
\usepackage{array}
\usepackage{hyperref}
\usepackage{color}
\usepackage{multirow}
\usepackage{tabularx}
\usepackage{tikz}
\usetikzlibrary{patterns}
\usetikzlibrary{decorations.pathreplacing}
\usepackage[utf8]{inputenc}

\hypersetup{colorlinks=false}

\makeatletter
\let\@fnsymbol\@arabic
\makeatother

\tikzset{every picture/.style={line width=0.75pt}} 
\usetikzlibrary{patterns}

\tikzset{
    pattern size/.store in=\mcSize, 
    pattern size = 5pt,
    pattern thickness/.store in=\mcThickness, 
    pattern thickness = 0.3pt,
    pattern radius/.store in=\mcRadius, 
    pattern radius = 1pt
}
\makeatletter
\pgfutil@ifundefined{pgf@pattern@name@_7gn8i7dj0}{
    \pgfdeclarepatternformonly[\mcThickness,\mcSize]{_7gn8i7dj0}
    {\pgfqpoint{0pt}{0pt}}
    {\pgfpoint{\mcSize+\mcThickness}{\mcSize+\mcThickness}}
    {\pgfpoint{\mcSize}{\mcSize}}
    {
    \pgfsetcolor{\tikz@pattern@color}
    \pgfsetlinewidth{\mcThickness}
    \pgfpathmoveto{\pgfqpoint{0pt}{0pt}}
    \pgfpathlineto{\pgfpoint{\mcSize+\mcThickness}{\mcSize+\mcThickness}}
    \pgfusepath{stroke}
    }
}
\makeatother

\if@twoside
\else 
\fi

\numberwithin{equation}{section}
\counterwithin{figure}{section}
\counterwithin{table}{section}

\theoremstyle{plain}
\newtheorem{theorem}{Theorem}[section]
\newtheorem{lemma}[theorem]{Lemma}

\theoremstyle{definition}

\theoremstyle{remark}

\newcommand{\norm}[1]{\left\|#1\right\|}
\newcommand{\abs}[1]{\left\vert#1\right\vert}
\newcommand{\spr}[1]{\left\langle\,#1\,\right\rangle}
\newcommand{\kl}[1]{\left(#1\right)}
\newcommand{\Kl}[1]{\left\{#1\right\}}

\newcommand{\R}{\mathbb{R}} 

\def\div{\mfunc{div}} 

\newcommand{\yd}{y^{\delta}}

\newcommand{\eps}{\varepsilon}

\newcommand{\uf}{\boldsymbol{u}}
\newcommand{\vf}{\boldsymbol{v}}
\newcommand{\wf}{\boldsymbol{w}}
\newcommand{\zf}{\boldsymbol{z}}
\newcommand{\hf}{\boldsymbol{h}}
\newcommand{\ff}{\boldsymbol{f}}
\newcommand{\yf}{\boldsymbol{y}}
\newcommand{\laf}{\boldsymbol{\lambda}}

\renewcommand{\div}{{\rm div}}

\let\div\undefined
\DeclareMathOperator{\div}{div}
\DeclareMathOperator*{\argmin}{arg\,min}
\newcommand{\Kon}{K_c}

\newcommand{\inter}{\text{\rm int}}

\newcommand{\comment}[1]{}
\newcommand{\crich}{\gamma_1}
\newcommand{\clow}{\underline{c}}
\newcommand{\cupper}{\overline{c}}

\title{On a PDE-based material parameter identification problem with contact constraints}
\author{
Simon Hubmer\footnote{Johannes Kepler University Linz, Institute of Industrial Mathematics, Altenbergerstra{\ss}e 69, A-4040 Linz, Austria, (simon.hubmer@jku.at)} ,
Stefan Kindermann\footnote{Johannes Kepler University Linz, Institute of Industrial Mathematics, Altenbergerstra{\ss}e 69, A-4040 Linz, Austria, (kindermann@indmath.uni-linz.ac.at)} ,
Ekaterina Sherina\footnote{Johann Radon Institute for Computational and Applied Mathematics (RICAM), Altenbergerstra{\ss}e 69, A-4040 Linz, Austria (ekaterina.sherina@ricam.oeaw.ac.at), \textbf{corresponding author}}
}
\date{\today}

\begin{document}

\maketitle

\begin{abstract}

We consider the identification of a scalar coefficient in a PDE-based parameter estimation problem with contact constraints. The considered problem can be used as an idealized model of a membrane under forces, constrained by a barrier or indenter. More generally, it serves as a benchmark for the analysis of more complex contact problems and the development of corresponding reconstruction algorithms. In this paper, we discuss both the forward and inverse parameter estimation problems, as well as uniqueness and non-uniqueness issues caused by the contact constraints. Furthermore, we consider the design and implementation of reconstruction approaches, which we test on numerical examples, illustrating both uniqueness and non-uniqueness as well as parameter identifiability.


\end{abstract}

\section{Introduction}\label{sect_introduction}

In this paper, we consider the estimation of the scalar coefficient $a$ in the constrained~PDE
    \begin{equation}\label{problem}
    \begin{split}
        - \div\kl{a(x)\nabla u(x) } &= f(x) \,, \qquad \forall \, x \in \Omega \cap \Kl{ u>h} \,, 
        \\
        u(x) &\geq h(x) \,, \qquad \quad \,\,\, \text{a.e.\ in } \Omega \,, 
    \end{split}    
    \end{equation}
augmented with boundary conditions (e.g., $u = 0$ on $\partial \Omega$), where $\Omega$ is a bounded Lipschitz domain, the functions $f$ and $h$ are known, and direct or indirect measurements of $u$ are available up to a certain amount of measurement noise. We furthermore assume in the following that $0 < \clow \, \leq a(x)  \, \leq \, \cupper$, a.e.\ in $\Omega$ with positive constants $\clow,\cupper \in \R$.

Without the constraint $u \geq h$, this is a standard parameter estimation problem for PDEs and has been studied extensively in the past, both theoretically and numerically, for example, within the context of the inverse conductivity problem \cite{CalderonAlbertoP2006Oaib, Isakov_2006}, electrical impedance tomography \cite{CheneyMargaret1999EIT,Mueller_Siltanen_2012,HolderDavid2022Eit}, or acousto-electrical tomography \cite{Zhang_Wang_2004,Capdeboscq_Fehrenbach_Gournay_Kavian_2009,Ammari_Bonnetier_Capdeboscq_Tanter_Fink_2008, Bal_Naetar_Scherzer_Scotland_2013, Hubmer_Knudsen_Li_Sherina_2018}. However, less is known about the parameter identification in the setting \eqref{problem} with constraints. The model \eqref{problem} is a classical obstacle/contact problem in the strong form and can be considered a free-boundary problem \cite{Friedman} for the contact area $\{u = h\}$. For the analysis below, we rely on its weak formulation, which is based on variational inequalities; see, e.g., \cite{KiSta}.

The model problem \eqref{problem} can be seen as an idealized model of a membrane under forces in the linear elasticity (i.e., small deformation) regime: We may formally derive the PDE in \eqref{problem} from 3D linear elasticity \cite{Ciarlet_1994} by postulating a displacement only in vertical direction $x_3$:  $u(x_1,x_2,x_3) = (0,0,u(x_1,x_2))$. The PDE  then represents the optimality condition for a minimization of the elastic energy 
    \begin{equation*}
        J(u) = \frac{1}{2} \int_\Omega a(x) |\nabla u(x)|^2 dx - \int f(x) u(x) \,,
    \end{equation*} 
where $a(x) = \int_{\text{thickness}} \mu(x_1,x_2,x_3) d x_3$, $\mu$ is the shear modulus, and $f$ is a vertical force density per area. Thus, a spatially varying coefficient $a(x)$ can  be both related to a inhomogeneous material distribution by the elastic modulus $\mu$  and to a nonuniform thickness. A rigorous general derivation of such types of energies for membrane models using $\Gamma$-convergence can be found in \cite[Theorem 1]{Friesecke}. 

The constraint in \eqref{problem} models a barrier or an indenter with  shape given by the function $h$. As such, it serves as a simplified model of an indentation test in elasticity/elastography. The goal of the inverse problem is to identify $a(x)$ from (partial knowledge) of a single or multiple measurements of the deformation $u$ in this idealized indentation test. Our inverse problem can be  seen as the simplest  benchmark problem for inverse contact problems and a first step towards a  further analysis of material parameter estimation problems with more complex contact conditions, such as the Signorini conditions \cite{KikuchiOden}, and for the development of corresponding reconstruction approaches.

In contrast to the problem without constraints, we found that the coefficient-inverse problem for \eqref{problem} (or related constraint problems) has been comparatively scarcely studied in the literature: Parameters (both elastic ones and others) were identified mostly by using variants of an output-least-squares fitting, e.g., in \cite{Sahin,Bensaada,MaLiu,SunLin}; the parameter identification for a 1D beam model was investigated in \cite{Radova}. Furthermore, the identification of elastic parameters from indentation curves has been studied in \cite{CoTa,TaCo} (see also the informative review \cite{BoCo}); finally, a parametric identification problem for indentation problems was considered in~\cite{Fernandez}. On a more abstract level, related to our model problem, there is considerable literature for inverse problems for variational inequalities, e.g., the comprehensive work of \cite{Gwinner2018,Gwinnerbook,Gwinner2,Gwinner1,KhanMot,KhanRac}. These results are broad extensions of the abstract energy-functional (variational) approach for elliptic parameter identification problems (see, e.g., \cite{Gockenbach_Jadamba_Khan_2008,Gock,GockII}) to the contact case.

Our work has  a slightly different focus: We allow from the onset the case of  partial observations, which does not fit completely into the abstract energy approach. This viewpoint is justified, since in some elastic indentation tests, the indenter geometry may obstruct access to knowledge of $u$ near the contact area such that full  measurements of $u$ cannot be reliably obtained in its immediate vicinity; additionally, as a consequence, we assume throughout that the contact area is an unknown. To compensate for lack of information, we allow from the beginning the case of multiple independent measurements (see, Section~\ref{sect_inverse}), that is, we may assume several independent data sets arising from experiments with different indenter shapes $h_i$ and forces $f_i$, $i = 1, \ldots K$, for the same unknown parameter $a$, giving rise to multiple independent solutions $u_i$. 

Moreover, instead of variational regularization by minimizing an energy functional, we investigate concrete minimization algorithms and in particular use iterative regularization methods of Landweber-type in finite- and infinite-dimensional settings. A particular difficulty that arises here is the problem of a lack of Fr\'{e}chet differentiability of the forward problem, requiring suitable approximation techniques. As a second, equally important aspect, we investigate uniqueness and non-uniqueness in the recovery of the coefficient $a$, for which we did not find specific results in the literature. In contrast to the classic inverse conductivity problem, non-uniqueness can be seen to occur due to the presence and shape of the indenter, and we identify several classes of non-uniqueness that may occur even for full-domain measurements. Finally, we conduct a number of numerical experiments on simulated data to illustrate the challenges of the problem, and to demonstrate the performance of the proposed reconstruction algorithms for simulated data. 

The main novel contributions of this article are the study of non-uniqueness as well as the development of an iterative regularization method for a potentially non-Fr\'{e}chet-differentiable problem using barrier methods, together with numerical case studies that illustrate convergence and practical applicability.

The outline of this paper is as follows: In Section~\ref{sect_forward} and Section~\ref{sect_inverse} we study the forward and inverse contact problem, respectively, and consider uniqueness and non-uniqueness caused by the contact constraints. In Section~\ref{sect_regularization}, we consider regularization approaches for the inverse contact problems, and in Section~\ref{sect_discrete} discuss the discretization of both the forward and inverse problems. Finally, in  Section~\ref{sect_numerics}, we provide a number of illustrative numerical examples, and finish with a brief conclusion in Section~\ref{sect_conclusion}.

\section{The forward contact problem}\label{sect_forward}

\begin{figure}[ht!]
    \begin{center}
        \includegraphics[trim = {30 35 40 12}, clip=true, width = 0.45\textwidth,height=0.47\textwidth]{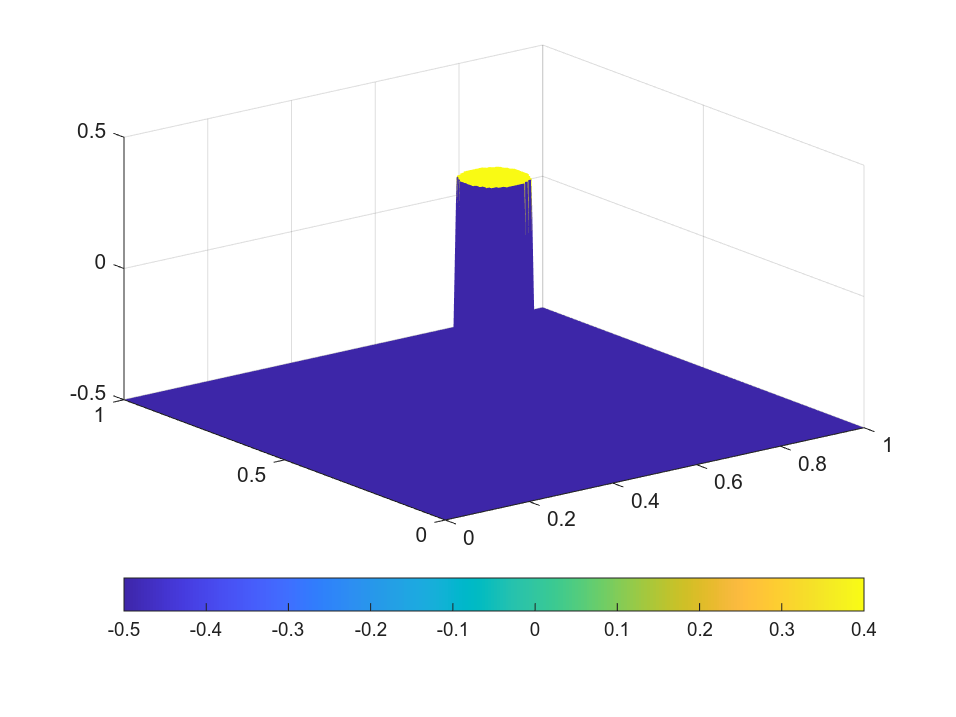}
        \includegraphics[trim = {30 35 40 12}, clip=true, width = 0.45\textwidth,height=0.47\textwidth]{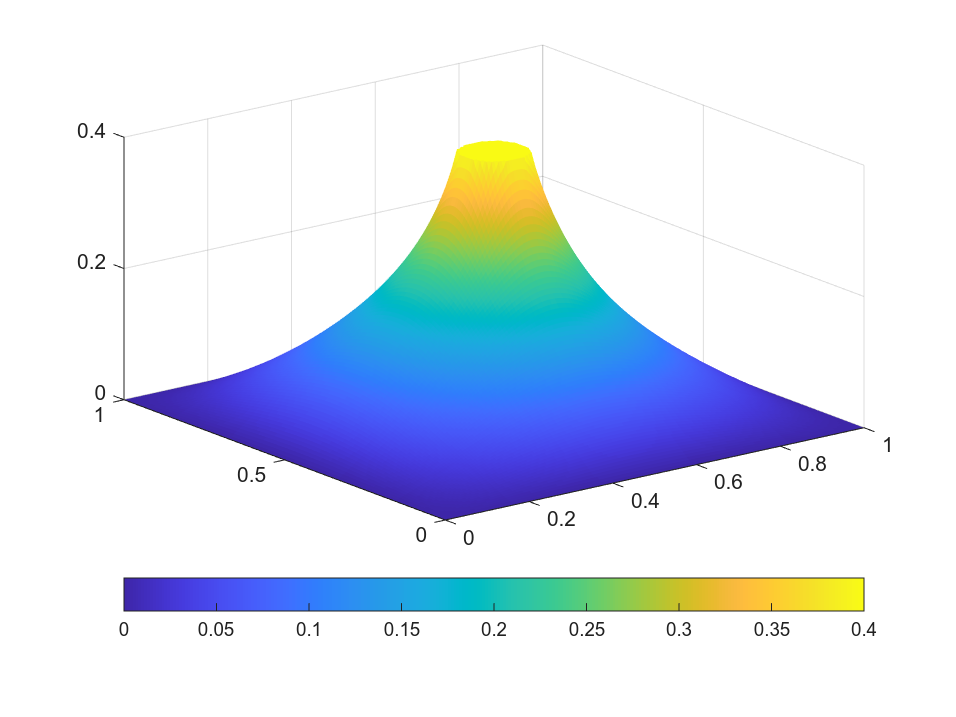}
    \end{center}
    \caption{Illustrative example of a constraint function ($=$ indenter shape) $h(x)$ (left) and the corresponding solution of the forward contact problem \eqref{eq:main} (right) (calculated on a $160\times 160$ grid).}
    \label{fig:zero}
\end{figure} 

\noindent
In this section, we consider the mathematically rigorous formulation of our model problem \eqref{problem} for $u$, and we discuss a number of equivalent variants used below. 

For this, we start by reformulating the problem as a constraint energy minimization problem in appropriate Hilbert spaces for given $a$ (and $f,h$, and, for simplicity, using homogeneous Dirichlet conditions on the boundary): The solution $u$ is defined by 
    \begin{equation}\label{eq:main}
    \begin{split}
        u &:= \argmin_{v \in H_0^1(\Omega)} \frac{1}{2} \int_\Omega a(x) \abs{\nabla v(x)}^2 \, dx - \int_\Omega f(x) v(x) \, dx \,,
        \\
        & \qquad \quad \text{ s.t.\ }\  v(x) \geq h(x) \qquad \text {a.e.\ in   } \Omega \,.
    \end{split}
    \end{equation} 
Figure~\ref{fig:zero} illustrates an instance of $h$ and the corresponding solution $u$. Intuitively, it is clear that the optimality conditions for the minimization problem in \eqref{eq:main} lead to the equations in \eqref{problem}, and for the remainder of the paper, we thus use \eqref{eq:main} as the (rigorous) definition of our model problem. Note that an alternative treatment of, e.g., Neumann boundary conditions could be pursued along similar lines.

First, we show that \eqref{eq:main} has a unique solution under standard assumptions.  For this, we start by expressing the constraint $v \geq h$ as $v \in K_h$, where 
    \begin{equation}\label{def:kon}
        K_h := \Kl{ v \in H_0^1(\Omega) \,\vert \, v(x) \geq h(x)  \text { a.e.\ in } \Omega } \,.
    \end{equation} 
With this definition, we obtain the following characterization:    

\begin{lemma}\label{lemma_Kh}
For $h \in H_0^1(\Omega)$ the set $K_h$ defined in \eqref{def:kon} is closed. Furthermore, 
    \begin{equation*} 
        K_h  = h + K_0 \,,
    \end{equation*}
and $K_0$ is a closed convex nonempty cone in $H_0^1(\Omega)$. 
\end{lemma} 
\begin{proof}
The set $K_h$ is nonempty, since $h$ is contained in it. Furthermore, it is clear that $K_0$ is a convex cone. Since any convergent sequence in $H_0^1(\Omega)$ has an a.e.\ pointwise convergent subsequence, it also follows that $K_0$ and thus $K_h$ are  closed. 
\end{proof} 

Next, we state the ingredients for a variational form of \eqref{eq:main}: For a fixed $a \in L_\infty(\Omega)$ with $a \geq \clow >0 $, let the bilinear form $b$ and the linear form $l$ be defined by 
    \begin{equation}\label{def_b_l} 
        b_a(u,v):=  \int_\Omega a(x) \nabla u(x)\cdot\nabla v(x) \, dx \,, \qquad \text{and} \qquad l(v) := \int_\Omega f(x) v(x) \, dx \,. 
    \end{equation}
With this, we can write \eqref{eq:main} as the convex unconstrained optimization problem  
    \begin{equation}\label{funct}
        u = \argmin_{v \in H_0^1(\Omega)} J_a(v) \,, 
        \qquad \text{with} \qquad 
        J_a(v) := \tfrac{1}{2} b_a(v,v) - l(v) + i_{K_h}(v) \,,
    \end{equation}
where $i_{K_h}$ denotes the indicator function 
    \begin{equation*}
        i_{K_h}(v) = 
        \begin{cases} 
            0 \,, & v \in K_h \,,
            \\ 
            \infty \,, & \text{else} \,.
        \end{cases}
    \end{equation*} 
Note that $i_{K_h}(v)$ is a convex functional, and its subgradient is the normal cone, i.e.,
    \begin{equation*} 
        \partial \, i_{K_h}(v) = N_{K_h}(v) := \Kl{ v^* \in H^{-1}(\Omega) \, \vert \, \spr{v^*,v- u} \leq 0 \,, \quad \forall \, v \in K_h } ,
    \end{equation*}
see, e.g., \cite[Def.~9.5.2]{att}. For later reference, and given a solution $u$ of \eqref{eq:main} or \eqref{funct}, we also define the contact area $\Kon$, i.e., the set where the constraints are active, as
    \begin{equation}\label{Kon}
        \Kon := \Kl{ x \in \Omega \, \vert \, u(x) = h(x)} \,,
    \end{equation} 
which has to be understood modulo nullsets, analogous to the definition of the essential support. Concerning the solvability of \eqref{eq:main}, we have the following standard results:

\begin{theorem}\label{th:exist}
Let $f \in H^{-1}(\Omega)$, $h \in H_0^1(\Omega)$, and $a \in L_\infty(\Omega)$ with $a \geq \clow$. Then, there exists a unique solution $u\in H_0^1(\Omega)$ of \eqref{eq:main}, which is uniquely characterized by the variational inequality 
    \begin{equation*}
       b_a(u,v-u) - l(v-u) \geq  0 \,, \qquad \forall \, v \in H_0^1(\Omega) \cap K_h \,.
    \end{equation*}
\end{theorem}
\begin{proof}
The existence and uniqueness of $u\in H_0^1(\Omega)$ follow by considering the transformation $\tilde{u} = u - h$, and noting that the minimization problem for $\tilde{u}$ is a minimization over a closed convex set $K_0$ with a bounded and $H_0^1(\Omega)$-continuous, strictly coercive functional. The coercivity itself follows from standard Poincar\'{e} inequality estimates. The optimality condition of this minimization problem can be stated as 
    \begin{equation}\label{opc}
        b_a(u,\cdot) - l(\cdot) + \partial \,  i_{K_h}(u) = b_a(u,\cdot) - l(\cdot) + N_{K_h}(u)  \ni 0 \,.
    \end{equation}
Considering $v^*:= b_a(u,\cdot) - l(\cdot)$ as an element in $H^{-1}(\Omega)$, it follows that $-v^* \in N_{K_h}(u)$, which yields the variational inequality. Conversely, since the problem is convex, the variational inequality yields the first-order optimality condition, which is sufficient. 
\end{proof} 

The optimality conditions \eqref{opc} may  be written in the form of Lagrange multipliers:

\begin{theorem}
Under the assumptions of Theorem~\ref{th:exist}, the solution $u \in H_0^1(\Omega)$ of \eqref{eq:main} is characterized by the following condition: There exists a $\lambda \in H^{-1}(\Omega)$ such that 
    \begin{alignat}{2}
        b_a(u,v)  &= l(v) + \spr{\lambda ,v } \,, &\qquad &\forall\, v \in H_0^1(\Omega) \,, \label{this1}
        \\
        u &\geq h \,, & \qquad &\text{a.e.\ in } \Omega \,,\label{this2}
        \\
        \spr{\lambda ,v -u} & \geq 0 \,,& \quad &\forall \, v \in K_h \,.  \label{this3}
    \end{alignat}
\end{theorem}
\begin{proof}
This follows from, e.g., \cite[Thm.~9.5.5]{att}. Note that the constraint qualification is satisfied, since the functional $ u \mapsto \frac{1}{2} b_a(u,u) - l(u)$ is continuous. 
\end{proof} 

Next, following \cite[p.~339]{att}, we rewrite the above problem into a KKT-type condition. For this, 
we first choose $v = h$ in \eqref{this3} and in a second step $v = 2 u -h \in K_h$ to find that $\spr{\lambda , u-h} = 0$. Inserting this into \eqref{this3} and setting $v  = h + w$, $w \geq 0$, we find that $\spr{\lambda ,w} \geq 0 $ for all $w \geq 0$. This is the weak form of the pointwise constraint $\lambda \geq 0$. Thus, we obtain the equivalent optimality conditions
    \begin{alignat}{2}
        b_a(u,v)  &= l(v) + \spr{\lambda ,v } \,, & \qquad &\forall \, v \in H_0^1(\Omega) \,, \label{cc1}
        \\
        u &\geq h \,, & \quad &\text{a.e.\ in } \Omega \,, \label{cc2}
        \\
        \spr{\lambda ,w} &\geq 0 \,, & \quad &\forall \, w \geq 0 \,, \label{cc3}
        \\
        \spr{\lambda ,h -u} &= 0 \,.& \label{cc4}
    \end{alignat}
Conversely, taking $v = h +w$, it follows that these conditions imply \eqref{this1}--\eqref{this3}, and thus they are equivalent. Furthermore, under the assumption that $\lambda$ possesses higher regularity than $H^{-1}(\Omega)$ and that it can be represented by an $L_2(\Omega)$-function, such that $\spr{\lambda ,v} = \int \lambda(x) v(x) \, dx$ and, if $u$, $a$, and $f$ are also sufficiently regular to be able to write \eqref{cc1} in its strong form, then the above conditions yield the classic KKT system
    \begin{alignat*}{2}
        -\div(a \nabla u) &= f + \lambda \,, &\quad &\text{in } \Omega \,,
        \\ 
        u & = 0 \,, & \quad &\text{a.e.\ on } \partial \Omega \,,
        \\
        \lambda & \geq 0 \,, & \quad  &\text{a.e.\ in } \Omega \,,
        \\
        u & \geq h \,, & \quad  &\text{a.e.\ in } \Omega \,,
        \\ 
        \lambda(u-h) &= 0 \,, & \quad  &\text{a.e.\ in } \Omega \,, 
    \end{alignat*}
which one can understand as the ``completion'' of \eqref{problem} required for well-definedness. Note that 
the Lagrange parameter $\lambda$ appears on the right-hand side of the equation
and  thus has the  interpretation  of a  contact force density.

Instead of a KKT system, the optimality condition \eqref{opc} may also be rewritten in a fixed-point form. For this, consider the projector onto the optimality constraint space:
    \begin{equation*}
        P_{K_h} y :=\argmin_{v \in K_h} \frac{1}{2} \norm{v - y}_{H_0^1(\Omega)}^2 = \argmin_{v \in H_0^1(\Omega)} \frac{1}{2} \norm{v - y}_{H_0^1(\Omega)}^2 +  i_{K_h}(v) \,, 
        \qquad \forall \, y \in H_0^1(\Omega) \,.
    \end{equation*}
The minimizer above exists due to standard arguments \cite{Bauschke_Combettes_2017}, and is characterized by
    \begin{equation}\label{HELP}
        P_{K_h} y = ( I  +\partial \, i_{K_h})^{-1} y \,, 
    \end{equation}
which is single valued. Note that if $y \in H^{-1}(\Omega)$, then $P_{K_h}y$ can still be well-defined by adapting the above minimization problem, and replacing $I$ in \eqref{HELP} with the embedding 
    \begin{equation*}
        \imath : H_0^1(\Omega) \to H^{-1}(\Omega) \,,
        \qquad
        y \mapsto \kl{v \mapsto \spr{y,v}_{H_0^1(\Omega)}} \,.
    \end{equation*}
With this, we can rewrite \eqref{opc} in a fixed-point form, namely
    \begin{equation*}
    \begin{split}
        &\imath  u + \tau \kl{ b_a(u,\cdot) - l(\cdot)} + \tau \partial \ i_{K_h}(u) \ni \imath  u 
        \\
        & \qquad \Longleftrightarrow \qquad
        \imath  u  + \tau \partial \ i_{K_h}(u) \ni \imath  u - \tau \kl{ b_a(u,\cdot) - l(\cdot)}
        \\
        & \qquad \Longleftrightarrow \qquad
        u = P_{K_h} ( \imath  u - \tau (b_a(u,\cdot) - l(\cdot))) \,,
    \end{split}
    \end{equation*}
for an arbitrary $\tau>0$, and where we note that $\tau \partial \, i_{K_h}(u) = \partial \, i_{K_h}(u)$. Applying a fixed-point iteration now yields the proximal gradient method (see, e.g., \cite{beck})
    \begin{equation}\label{projgrad} 
        u_{n+1} = P_{K_h} \kl{\imath u_n - \tau (b_a(u_n,\cdot) - l(\cdot))} \,, 
    \end{equation}
which can be used to solve the forward problem. The drawback of this formulation is the need for the $H_0^1(\Omega)$-projection instead of the convenient $L_2(\Omega)$-projection, given by 
    \begin{equation*}
        P_{K_h}^{L_2}y = h + \max(y-h,0)\,.
    \end{equation*}
However, this drawback can be remedied in a finite-dimensional context; cf.~Section~\ref{sect_discrete}.

\vspace{10pt}\noindent
\textbf{Boundary constraints:} Note that instead of inner constraints caused by an indenter, one can also consider an analogous setting to \eqref{problem} or \eqref{eq:main} with boundary constraints. For example, assume that we are given a contact area $\Gamma_c \subset \partial \Omega$ on which we have the constraint $u(x) \geq h(x)$, and let, for simplicity, $u=0$ on $\Gamma_0:= \Omega \setminus \overline{\Gamma_c}$. Furthermore, let
    \begin{equation*}
        H_{0,\Gamma_0}^1(\Omega):= \Kl{ v \in H^1(\Omega) \,\vert\, v=0 \mbox{ on } \Gamma_0 }\,,
    \end{equation*}
and, given $h \in H_{0,\Gamma_0}^1(\Omega)$, define the modified constraint set as 
    \begin{equation*}
        K_h^{\operatorname{bnd}} := \Kl{ v \in H^{1}(\Omega) \,\vert\, v(x) \geq h(x) \mbox{ a.e.\ in } \Gamma_c }\,.
    \end{equation*}
Analogously to \eqref{funct}, we then consider the following contact problem in functional form:
    \begin{equation}\label{funct_bnd}
        u = \argmin_{v \in H_{0,\Gamma_0}^1(\Omega)} J_a^{\operatorname{bnd}}(v) \,, 
        \qquad \text{with} \qquad 
        J_a^{\operatorname{bnd}}(v) := \tfrac{1}{2} b_a(v,v) - l(v) + i_{K_h^{\operatorname{bnd}}}(v) \,.
    \end{equation}
By continuity of the trace operator, the set $K_h^{\operatorname{bnd}}$ is again a closed convex affine cone. Hence, all of the above characterizations for the inner constraint problem \eqref{eq:main} or \eqref{funct} also hold in a similar form for \eqref{funct_bnd} with a Lagrange multiplier in $H^{-1/2}(\Gamma_c)$, which now has the interpretation of an additional Neumann boundary condition on $\Gamma_c$.

\section{The inverse contact problem}\label{sect_inverse}

In this section, we consider the inverse contact problem of finding the coefficient $a$ from either a single or multiple (direct or indirect) measurements of $u$. We define

\vspace{10pt}\noindent
\textbf{The inverse contact problem:} Given $f_i \in H^{-1}(\Omega)$ and $h_i \in H_0^1(\Omega)$, for $i = 1,\ldots K$, as well as observation operators $B_i :H^1(\Omega) \to L_2(\Omega_i')$, with $\Omega_i' \subseteq \Omega$, find the coefficient $a$ from the data 
    \begin{equation*}
        y= \kl{B_i u_i}_{i=1}^K \,,
    \end{equation*}
where $u_i$ is a solution of \eqref{eq:main} with coefficient $a$ and given forces $f_i$ and constraints $h_i$. 

\vspace{10pt}

In case that the observation operator $B$ is the identity and $\Omega_i' = \Omega$, we have full (internal) measurements of $u_i$ on $\Omega$. Next, we define the parameter-to-data map
    \begin{equation}\label{eq:defF}
    \begin{split}
        F: D(F) \subset X &\to \prod_{i=1}^K L_2(\Omega'_i) \,, 
        \\ 
        a &\to \kl{ B_i u_i }_{i=1}^K \,,
    \end{split} 
    \end{equation}   
for some $X \subseteq L_\infty(\Omega)$, where
    \begin{equation*}
        D(F) := \Kl{ a \in X \,\vert\, \clow \leq a(x) \leq \cupper \,\text{ a.e.\ in } \Omega } \,.
    \end{equation*}
With this, the inverse contact problem can be written in the standard form
    \begin{equation}\label{Fa=y}
        F(a) = y \,,
    \end{equation}
which is a generalization of the well-known coefficient inverse problem of finding the coefficient $a$ in the equation $-\div (a \nabla u) = f$ from full or partial knowledge of~$u$. This classic coefficient inverse problem and the inverse contact problem are related as follows: 

\begin{itemize}
    \item If the Lagrange multiplier $\lambda$ were known in all instances, then the contact problem could be treated without any constraints by just considering the PDE, thus reducing to the classical problem. However, $\lambda$ is unknown, cannot be measured, and depends implicitly on the unknown $a$, making the contact problem more difficult. 
    \item If the contact area $\Kon$ defined in \eqref{Kon} is known (and sufficiently regular), then the (inverse) contact problem reduces to a boundary value problem in $\Omega \setminus \Kon$, where $u$ takes the Dirichlet boundary condition $h$ at $\partial \Kon$. In this case, we are again facing a standard coefficient problem in $\Omega \setminus \Kon$. However, in this paper, we are reluctant to assume knowledge of $\Kon$, since this is only possible when considering full measurements of $u$ (which may additionally be contaminated by noise) and, as discussed, might not be available due to obstruction by the indenter.
\end{itemize}

Another key difference between the classic coefficient inverse problem and the inverse contact problem is non-uniqueness issues for the reconstruction of $a$ that arise due to the constraints, which we investigate in the subsequent section.

\subsection{Uniqueness and non-uniqueness for the constrained and unconstrained inverse problems}

First, we recall some known standard uniqueness and non-uniqueness results for the unconstrained (i.e., without the constraint $u\geq h$) coefficient inverse problem: Uniqueness of $a$ given full measurements of $u$ was shown in \cite{Ri} (see also \cite{KoVo}) if the (Richter) condition $\max\Kl{\abs{\nabla u},\Delta u}>\crich >0$ holds in $\Omega$ and if $a$ is known on a relevant part of the boundary. Furthermore, in \cite{Ales} uniqueness was established for the Dirichlet problem in two dimensions if the coefficient $a$ is known on the boundary and the Dirichlet data have a finite number of maxima and minima. Finally, the authors of \cite{ChiconeGerlach} have shown uniqueness for the homogeneous problem on a set of characteristic lines. 

In the case of multiple measurements, \cite{Bal_Uhlmann_2013} has established uniqueness (up to a multiplicative constant) for the homogeneous problem and in case of Dirichlet boundary conditions from full measurements of $\frac{n (n+3)}{2}$ solutions $u_i$ (the anisotropic case was treated in \cite{Bal_Monard_Uhlmann2015}). In the absence of such conditions, the non-uniqueness of the coefficient is clear: As a simple example, take in 2D a solution $u(x,y) = c x + d$, $c,d \in \R$. Then any function depending only on $y$ can be added to the coefficient $a$ without affecting this solution. 

The above uniqueness conditions can immediately be transferred also to the constrained problem, but only in the non-contact region, where the PDE holds. However, the contact area adds some additional difficulty, since we can show that the coefficient $a$ cannot be recovered in the interior $\inter(\Kon)$ of the contact area $\Kon$. In particular, we have the following non-uniqueness result: 

\begin{theorem}\label{th3.1}
Let $a \in C^1(\Omega)$, $f \in C(\Omega)$, $h\in C^2(\overline{\Omega})$, and let $u_a$ be a solution of the contact problem \eqref{eq:main} with given $a,f,h$. Furthermore, let the active constraint set $\Kon$ corresponding to the solution $u_a$ be as in \eqref{Kon} and define the set
    \begin{equation*}
        D: = \inter\kl{\Kl{ x \in \inter(\Kon) \,\vert\, \nabla h(x) = 0 }} \cup \Kl{ x \in \inter(\Kon) \,\vert\, \div (a \nabla h(x))  \neq f(x)  } \,.
    \end{equation*}
Then, there exists an $\tilde{a}(x): = a(x) +\delta a \in D(F)$, where $\delta a \in C^1(\Omega)$ is supported in a compact subset of $D$ such that $u_a$ is also a solution of \eqref{eq:main} with the modified coefficient~$\tilde{a}$. In particular, $a$ is not uniquely determined by $u_a$ in any compact subset of~$D$.
\end{theorem} 
\begin{proof}
First, consider a coefficient perturbation in the interior of the contact area, i.e., $\tilde{a}(x) := a(x) +  \delta a(x)$ for some $\delta a(x)$, with
    \begin{equation*}
        \text{supp}(\delta a) \subset \inter(\Kon)
        \qquad \text{and} \qquad \delta a \in L_\infty(\Omega) \cap C^1(\inter(\Kon)) \,,
    \end{equation*}
the precise form of which is specified in detail below. We now show that such a perturbation does not influence the solution $u_a$. For this, let the bilinear form $b_a$ be given as in \eqref{def_b_l}. The optimality conditions for the given solution $u_a$ and the associated Lagrange multiplier $\lambda_a$ and for $\tilde{a}$ are
    \begin{equation*} 
        b_{\tilde{a}}(u,v) = b_a(u,v) +  b_{\delta a}(u,v) = l(v) + \spr{\lambda_a v } +  b_{\delta a}(u,v) \,,
    \end{equation*}
in addition to the conditions \eqref{cc2}--\eqref{cc4}. Since $\delta a$ has support in $\Kon$, we have $ b_{\delta a}(u,v) = b_{\delta a}(h,v)$. Next, define the Lagrange multiplier
    \begin{equation*}
        \tilde{\lambda}:= \lambda_a +  b_{\delta a}(h,\cdot)\,,
    \end{equation*}
which represents an element in $H^{-1}(\Omega)$. With this $\tilde{\lambda}$, the first condition, \eqref{cc1}, holds. Condition~\eqref{cc4} holds since $u=h$ at the support of $\delta a$, and thus $ b_{\delta a}(h,u-h) = 0$.
 
Hence, we now only have to verify \eqref{cc3}. For this, note first that due to the smoothness assumptions, the set of all $x$ in $\inter(K)$ with $\div (a \nabla h(x)) \neq f(x)$ is open, and thus $D$ is open. Furthermore, again due to the smoothness assumptions, we have
    \begin{equation*}
        \lambda_a = -\div(a\nabla u) - f =  -\div(a\nabla h) - f \,,
        \qquad 
        \text{in} \,\, \inter(K) \,, \,
        \text{pointwise in} \, D \,, 
    \end{equation*}
as well as $\tilde{\lambda} = \lambda_a - \div(\delta a \nabla h)$. If $x \in \inter(\Kl{ x \in \inter(\Kon) \,\vert\, \nabla h(x) = 0 }) $, it follows that $\tilde{\lambda} = \lambda_a$, and thus positivity of $\tilde{\lambda}$ is inherited from $\lambda_a$, establishing \eqref{cc3} in this case.  

On the other hand, for all $x \in  \tilde{D}:= \inter(\Kl{ x \in \inter(\Kon) \,\vert\, \div (a \nabla h(x) ) \not = f(x)  })$, we have $\lambda_a \not =  0$ by the pointwise equality, and thus $\lambda_a  > 0$ in this set. Hence, for an arbitrary compact subset, we have $\lambda_a(x) \geq d_1>0$ by the Heine theorem. Now take an arbitrary $\phi \in C^1(\Omega)$ with compact support in $\tilde{D}$ and set $\delta a := \eps \phi$ with $\eps > 0$. Then 
    \begin{equation*}
        \tilde{\lambda} = \lambda_a - \epsilon \div(\phi \nabla h) \geq d_1 - \epsilon \norm{\phi}_{C^1} \norm{h}_{C^2} \,.
    \end{equation*}
By choosing $\epsilon$ small enough, we find $ \tilde{\lambda} \geq 0$, verifying \eqref{cc3} and completing the proof. 
\end{proof}

The above theorem roughly states that we cannot identify the coefficient $a$ at the contact area. (More precisely, formally this only holds on the contact area where $\nabla h = 0$ or where $\lambda_a \not = 0$. However, the case that this does not agree with $K_c$ represents an exceptional case which corresponds to areas where no force is required to keep contact; thus Theorem~\ref{th3.1} can be (informally) considered to hold for the full contact area $K_c$.) Note that at the exceptional set $K_c\setminus D$, one could identify $a$ by using the equation $\div(a \nabla h) = f$ up to boundary conditions of $a$ as in the classical case. 

In addition to the non-uniqueness in the contact area above, we can illustrate further examples of non-uniqueness even away from the contact area: this is related to the non-connectedness of the non-contact area and to the lack of the Richter condition (see above) therein. Note that for the homogeneous problem without constraints, the Richter conditions are not too restrictive, since the gradient and the Laplacian cannot vanish on an open set, which allowed the uniqueness result in \cite{Ales} discussed above. In contrast, this might not be the case for some standard constrained problems:

\begin{figure}[ht!]
    \begin{center}
        \includegraphics[width=0.4\textwidth,height=0.405\textwidth]{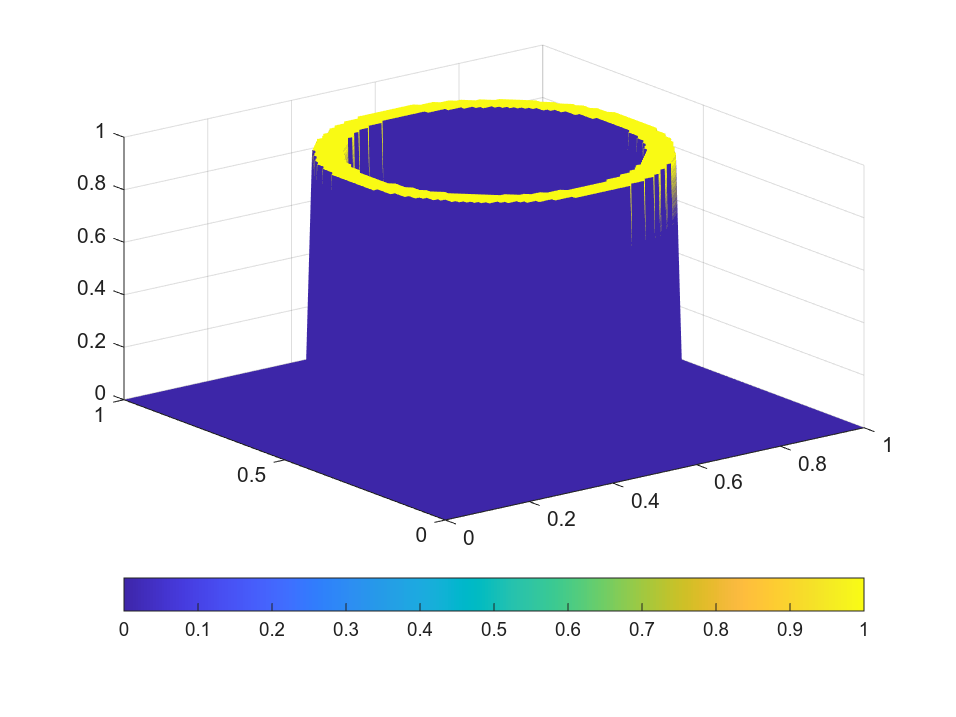}
        \includegraphics[width=0.4\textwidth,height=0.405\textwidth]{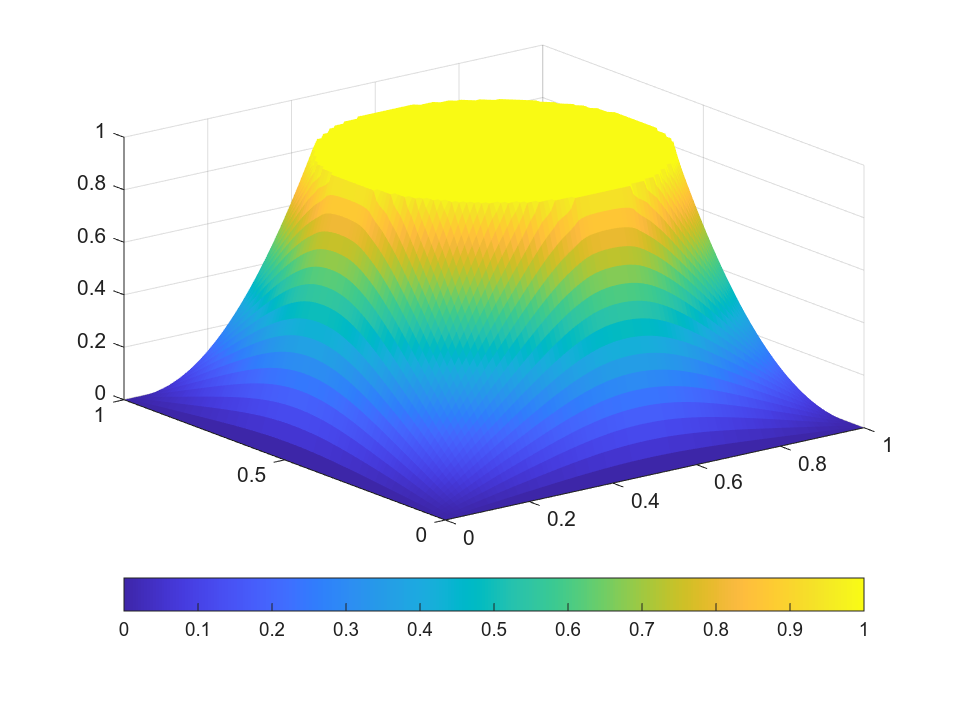}
    \end{center}
\caption{Left: constraint function $h$. Right: solution of $\div(a \nabla u) = 0$ (calculation on a $100 \times 100$ grid). Inside the annulus contact region, the coefficient $a$ is not identifiable.}\label{fig:ann}
\end{figure}

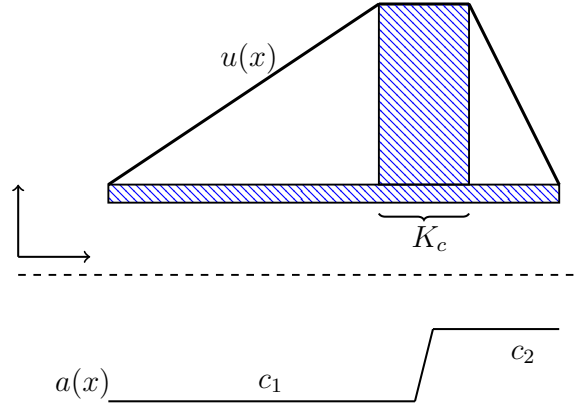
\begin{figure}[ht!]
\centering
\resizebox{0.5\textwidth}{!}{%
\begin{tikzpicture}
\tikzstyle{every node}=[font=\normalfont]
    \node at (5.7,11.0) {$u(x)$};
\draw[very thick] (3.75,9.25) -- (7.5,11.75);
\draw[very thick]   (7.5,11.75) -- (8.75,11.75);
\draw[very thick]   (8.75,11.75) -- (10,9.25);
\draw [->,thick] (2.5,8.25) -- (2.5,9.25);
\draw [->,thick] (2.5,8.25) -- (3.5,8.25);
\draw [dashed] (2.5,8.) -- (10.5,8.);
\draw[pattern=north west lines, pattern color=blue] (7.5,9.25) rectangle (8.75,11.75);
\draw[pattern=north west lines, pattern color=blue] (3.75,9.) rectangle (10,9.25);
\draw [thick,decorate,
    decoration = {brace,mirror}] (7.5,8.85) --  (8.75,8.85);
\node at (8.2,8.5) {$K_c$};   
    
  \begin{scope}[shift={(0,-4)}]
    \node at (3.4,10.5) {$a(x)$};
  \node at (6.0,10.5) {$c_1$};
    \node at (9.5,10.9) {$c_2$};
\draw [thick] (3.75,10.25) -- (8.0,10.25); 
\draw [thick] (8.25,11.25) -- (10.0,11.25) ; 
\draw [thick] plot [smooth] coordinates { (8.0,10.25)    (8.25,11.25)  };
\end{scope}
\end{tikzpicture}
}%
\caption{Illustration of a non-unique parameter $a$ in a 1D contact problem: The solution is piecewise constant over the constraint 
given by $h$ (blue dashed object). Below is an illustration of the non-uniqueness of $a$: 
It is piecewise constant with smooth connection inside the contact area. The constants $c_1,c_2>0$ can 
be chosen arbitrarily.}
\label{fig:one}
\end{figure}

\begin{enumerate}[label=(\roman*)]
    \item The contact problem can induce further non-uniqueness issues even outside of the contact area $K_c$, e.g., if the contact area is not connected: Assume $f = 0$ and let $K_c$ be an annulus as in Figure~\ref{fig:ann}. Then, we have non-uniqueness at the annulus contact area but also \emph{inside} of it, since there $u = \text{const.}$ and thus $\nabla u = \Delta u = 0$.
    \item In Figure~\ref{fig:one}, we illustrate another instance of non-uniqueness due to the violation of the Richter condition in a 1D example. Here, $f = 0$ and the solution $u$, $x \in [0,1]$, satisfies the equation $-(a u')' = \lambda$, with the constraint given by $h$ indicated as the upper boundary of the blue dashed object (indenter). The graph of $u$ is printed as a thick black line for the case when $a$ is a constant: the solution is piecewise constant with contact area $K_c$. The Lagrange multipliers are $\delta$-distributions supported on the boundary of the contact area. However, it turns out that the same solution $u$ is obtained if $a$ is piecewise constant with constants $c_1$ and $c_2$ and a smooth connection inside the contact area as indicated in the figure. Thus, non-uniqueness in $a$ is apparent and occurs here due to the non-connectedness of the non-contact area. 
    
Note that this example models a string spanned over a fixed bridge, and intuitively it is clear that changing the  stress (i.e., $a$) or the material uniformly does not alter the shape of the string. A similar example can be found in a 2D setting by extending $u$ constant into $y$-direction and modifying the boundary conditions in \eqref{eq:main} to homogeneous Neumann boundary conditions at $y = 0$ and $y = 1$.

\end{enumerate}

The last example of non-uniqueness can be generalized to the following theorem:

\begin{theorem}\label{th:mulit}
Let $u$ satisfy the homogeneous model problem~\eqref{problem} (i.e., with $f=0$ and Dirichlet boundary data $g$). Furthermore, assume that the non-contact area $\Kl{u\not = h}$ splits into $n$ disjoint regions $\Omega_i$, $i =1,\ldots,n$. Then, there are $n$ constants $c_i$ such that the coefficient $a$ can be multiplied by $c_i$ on $\Omega_i$ giving rise to the same solution $u$. 
\end{theorem}
\begin{proof}
On each of the regions $\Omega_i$, $u$ solves a Dirichlet problem with Dirichlet data either $g$ or $h$. Since there $a$ can  be multiplied by a constant without altering the PDE, the result immediately follows.
\end{proof}

It should be noted that, for instance, in the annulus problem (but also for the problem in Theorem~\ref{th:mulit} as long as the contact area remains the same), the non-uniqueness is not necessarily resolved by multiple measurements: If the contact at the annulus is always active, then inside the annulus, the multiple solutions are always constant and thus linearly dependent. The non-uniqueness problem remains. Thus, to improve identifiability for multiple measurements, one should design the experiments (i.e., by altering $f$ or $h$) such that the contact area changes (see the last example in Section~\ref{sec6.2}).

\section{Regularization approaches}\label{sect_regularization}

In this section, we consider regularization approaches for the inverse contact problem. Since its formulation given in \eqref{Fa=y} is in standard form, we can use standard methods \cite{Engl_Hanke_Neubauer_1996} such as Tikhonov regularization, which computes an approximate solution as
    \begin{equation}\label{Tikh_cont}
        a_{\alpha} := \argmin_{a \in D(F)}  \norm{F(a) - y}^2 + \alpha \norm{a}_X^2 \,,
    \end{equation} 
with $\alpha$ a regularization parameter (see, e.g., \cite{Engl_Hanke_Neubauer_1996}). There exist many variations of this approach: the penalty term $\norm{a}_X^2$ may be replaced by general Banach space penalties (e.g., \cite{HofmannKaltenbacher07}) or the data-fit term by more general expressions (e.g., \cite{Scherzerbook}). Furthermore, for the non-constrained problem, it became popular to include the PDE weakly  into the functional as in the all-at-once approach \cite{Kalt,Kaltall,KaltVex}. By tailoring the data-fit term, one can establish the energy-functional variational formulation leading to convex optimization problems; for the unconstrained problem, this has been put into an abstract framework in \cite{Gock,GockII}; for the constrained problem---more precisely for the variational inequality case and for \emph{full} observations---this has been extended in \cite{Gwinner2018,Gwinnerbook,Gwinner2,Gwinner1,KhanMot,KhanRac}. Note that, since we here treat the limited data case, these latter approaches cannot be straightforwardly extended to our case. Hence, the most flexible choice here is standard Tikhonov regularization and related iterative regularization schemes such as variants of gradient descent for the least-squares functional with appropriate stopping rules. 

In order to verify the standard assumptions required for the convergence analysis of Tikhonov regularization, we first need to analyze  the parameter-to-solution mapping: 

\begin{theorem} 
Let $f_i \in H^{-1}(\Omega), h_i \in H_0^1(\Omega)$, $i=1,\dots,n$, and $F$ be defined as in~\eqref{eq:defF}, where $X$ is a Hilbert space compactly embedded into $L_\infty(\Omega)$. Assume that $B_i$ is continuous from the weak $H^1(\Omega)$ topology into $L_2(\Omega_i')$. Then $F$ is weakly sequentially closed.
\end{theorem}
\begin{proof}
Without loss of generality, we can assume that $n =1$, and thus write $f,h,u,B$ instead of $f_i,h_i,u_i,B_i$ throughout this proof. Now for any $a \in D(F)$, let $u_a \in H_0^1(\Omega)$ be the well-defined minimizer of $J_a$ defined in \eqref{funct}. Assume that $a_n \rightharpoonup_X a$  and $F(a_n) \to z$. Then $a \in D(F)$, and with the Poincar\'{e} inequality, it follows that
    \begin{equation*} 
        \clow \norm{u_{a_n}}_{H^1}^2 \leq \int_{\Omega} a_n(x) \abs{\nabla u_{a_n}(x)}^2 \, dx  \leq \int_{\Omega} a(x) \abs{\nabla u_{a_n}(x)}^2 \, dx  + i_{K_h}(u_{a_n}) \leq  \cupper \int_{\Omega} \abs{\nabla h} \, dx \,. 
    \end{equation*}
Hence, $u_{a_n}$ is uniformly bounded in $H^1(\Omega)$ and has a weakly convergent subsequence 
    \begin{equation*}  
        u_{a_n} \rightharpoonup_{H^1} w\,.
    \end{equation*} 
It is not difficult to show that $J_a$ is lower semicontinuous in $H^1(\Omega)$: This is obvious for the first part $(\tfrac{1}{2} b_a(v,v) - l(v))$, and the lower semicontinuity of $i_{K_h}(u)$ with respect to the norm convergence in $H^1(\Omega)$ follows from the fact that $K_h$ is closed in $H^1(\Omega)$; see Lemma~\ref{lemma_Kh}. Since $J_a$ is also convex, it is weakly lower semicontinuous. Furthermore,
    \begin{align*}
        J_a(u_{a_n}) &= \frac{1}{2} \int_{\Omega} a(x) \abs{\nabla u_{a_n}(x)}^2 \, dx - \int_{\Omega} f(x) u_{a_n}(x) \, dx + i_{K_h}(u_{a_n}) 
        \\
        & =  \frac{1}{2}\int_{\Omega} a(x) \abs{\nabla u_{a_n}(x)}^2 \, dx - \int_{\Omega} f(x) u_{a_n}(x) \, dx 
        \\
        &= \frac{1}{2} \int_{\Omega} a_n(x) \abs{\nabla u_{a_n}(x)}^2 dx - \int_{\Omega} f(x) u_{a_n}(x) \, dx  
        \\
        & \qquad + \frac{1}{2}
        \int_{\Omega} (a(x)-a_n(x)) \abs{\nabla u_{a_n}(x)}^2 \, dx   
        \\
        & \leq J_{a_n}(v)  + \frac{1}{2} \int_{\Omega} (a(x)-a_n(x)) \abs{\nabla u_{a_n}(x)}^2 \,dx 
        \\
        & \leq J_{a}(v) \frac{1}{2} \int_{\Omega} (a_n(x)-a(x)) \abs{\nabla v(x)}^2 \, dx + 
       \frac{1}{2}  \int_{\Omega} (a(x)-a_n(x)) \abs{\nabla u_{a_n}(x)}^2 \, dx  \,,
    \end{align*}
for any $v \in H_0^1(\Omega)$ due to the minimality property of $u_{a_n}$. Now since due to the compact embedding of $X$ in $L_\infty(\Omega)$ there holds $\norm{a_n-a}_\infty \to 0$, it follows that
    \begin{equation}\label{eq:open}
    \begin{split}
        \int_{\Omega} (a(x)-a_n(x)) \abs{\nabla u_{a_n}(x)}^2 \, dx \leq \norm{a_n-a}_\infty   \int_{\Omega}  \abs{\nabla u_{a_n}(x)}^2 \, dx   \leq C \norm{a_n-a}_\infty \to_{n \to \infty}  0 \,.
    \end{split}
    \end{equation}
Similarly, we obtain
    \begin{equation*}
        \int_{\Omega} (a_n(x)-a(x)) \abs{\nabla v(x)}^2 \, dx  \to_{n \to \infty}  0 \,.
    \end{equation*}
Combining the above, we find that for all $v \in H_0^1(\Omega) \cap K_h$ there holds
    \begin{equation*}  
        J_a(w) \leq \liminf_n  J_a(u_{a_n}(x))  \leq  J_{a}(v) \,.
    \end{equation*}
Hence, $w$ is a minimizer and thus by uniqueness, we obtain $w = u_a$.   
\end{proof}

Due to the above theorem, the classic convergence analysis of Tikhonov regularization is applicable, which yields the usual convergence results \cite{Engl_Kunisch_Neubauer_1989,Engl_Hanke_Neubauer_1996}. However, it does not tell us how to calculate the regularized solution $a_\alpha$, i.e., how to minimize \eqref{Tikh_cont}. Instead, for the numerical computations of least-squares or Tikhonov-regularized solutions, one has to rely on, e.g., iterative regularization (respectively, optimization) methods. The simplest representative of them is Landweber iteration (e.g.\ \cite{Engl_Hanke_Neubauer_1996,Kaltenbacher_Schoepfer_Schuster_2009,Kaltenbacher_Neubauer_Scherzer_2008}), which for the problem \eqref{eq:defF} is defined by 
    \begin{equation}\label{landwb}
        a_{k+1} = a_k - \tau F'(a_k)^*(F(a_k) - y^\delta) \,,
    \end{equation}
with $F'(a_k)^*$ denoting the adjoint of the Fr\'{e}chet derivative. Now this is the main algorithmic issue for the considered inverse contact problem: For Landweber's method, as well as almost all similar iterative regularization schemes such as (iteratively regularized) Gauss-Newton  or Levenberg-Marquardt-type methods, the forward operator has to be Fr\'{e}chet differentiable. This requirement most probably does not hold for the inverse contact problem (although we do not have an explicit proof of it).

There are strong indications that our forward operator $F$ is non-differentiable: We note that for some related problems in control theory, for instance, for the mapping $f \to u$ in the constrained case, only generalized differentiability results (e.g., semismoothness) have been established  in \cite{Rauls,Christof1, Christof2}. In particular, \cite[Theorem~5.6]{Rauls} indicates that in this case the generalized derivatives are typically multi-valued, which implies that classical differentiability cannot hold. Since the linearization of the parameter-to-solution map leads to exactly such mappings (with $-\div(h \nabla u)$ as  right-hand side), this should serve as a convincing argument that our inverse problem is also not classically differentiable. Verifying, e.g., semismoothness for $F$ in appropriate spaces would be a non-trivial investigation far outside the scope of the present paper. 

Note that iterative regularization schemes with generalized derivatives (e.g., Bouligand derivatives) have been developed for inverse problems \cite{Clason1, Rommel, JinBouligand}, but since we are unsure if our problem allows for such derivatives (in the required spaces!) we do not make use of them here. Potentially, the SCD semismooth* Newton method could be used for minimizing \eqref{Tikh_cont}, since it only requires weak differentiability conditions \cite{Gfrerer_Hubmer_Ramlau_2025,Gfr24,GfrOut21,GfrManOutVal22}.

In the following, we introduce our proposed iterative regularization method for solving the inverse model problem, which essentially consists of a variant of Landweber iteration \eqref{landwb} with a modified derivative in place of $F'(a_k)$. The idea of using an approximate derivative is very similar to the method \cite{Gwinner2018} for the energy method (full-data case). Additionally, we use Nesterov acceleration to speed up the reconstruction process \cite{Nesterov_1983}. Our reasoning for this choice is as follows: Of course, for a fast iteration, one could try to use second-order methods (such as Gauss-Newton, IRGN, Levenberg-Marquardt, etc.) as mentioned above. However, two reasons speak against it: Firstly, these methods -- in their exact form -- require the solution of a large linear system which we would like to avoid for reasons of efficiency and simplicity.  (One could, of course, use inexact Newton methods such as \cite{Rieder}, but this adds another level of approximation schemes which we would like to avoid). Secondly, since in our proposed approach we anyway replace the Fr\'echet-derivative by an approximation, we expect that second-order methods are more sensitive to taking this wrong derivative. (Partly, this can be seen from the fact that the convergence theory of several of these methods rely on  range-invariance rather than the tangential cone-type conditions \cite{Kaltenbacher_Neubauer_Scherzer_2008}.) 

Thus, we propose to use a Nesterov-accelerated Landweber method for solving the inverse model problem \eqref{Fa=y}, which we found to be a robust and efficient regularization scheme. We describe this method in detail in the following section and, to avoid substantial technical difficulties, only in a finite-dimensional (discretized) setup.

\section{The finite-dimensional problems}\label{sect_discrete}

In this section, we consider the forward and inverse contact problems in a discretized setting, and discuss different (iterative) methods for the efficient solution of the direct problem, namely a projected gradient method with its accelerated variant and a barrier method. Although other methods are certainly possible, the rationale for this choice is their simplicity and, as we shall see, their integrability into the regularization schemes. 

\subsection{The discretized forward contact problem}

In order to discretize the forward contact problem, we consider a continuous, piecewise linear finite-element nodal basis over a (uniform) triangulation of the domain $\Omega$, i.e., with $\phi_i(x)$, $i = 1,\ldots N$, denoting the nodal basis functions, 
    \begin{equation*}
        u = \sum_{i=1}^N u_i \phi_i\,,
        \qquad
        f = \sum_{i=1}^N f_i \phi_i \,,
        \qquad \text{and} \qquad
        h: = \sum_{i=1}^N h_i \phi_i \,.
    \end{equation*}
The minimization problem \eqref{eq:main} is then considered over the corresponding finite element space, and the constraint $u \geq h$ is replaced by nodal constraints on the FEM nodes, i.e.,
    \begin{equation*} 
        u_i \geq h_i \,, 
        \qquad \forall \, i =1, \dots, N \,.
    \end{equation*}
The resulting minimization problem can again be rewritten as in the continuous setting. In this case, the Lagrange multiplier $\lambda$ then corresponds to a vector $\laf = (\lambda_i)_{i=1}^N$, and the KKT conditions are given in the usual form: Let $A$ denote the stiffness matrix, $\ff$ the load vector, and $\uf,\hf$ the coefficient vectors of the corresponding basis functions. Then the KKT system of the discretized contact problem has the form: Find $\uf,\laf$ with 
    \begin{alignat}{2} 
        A \uf &= \ff + \laf \,, & 
        \\ \label{disc:one}
        u_i &\geq h_i \,, & \qquad &\forall \, i = 1\,,\dots \,, N \,, 
        \\
        \lambda_i &\geq 0 \,, & \qquad &\forall \, i= 1\,,\dots \,, N\,, 
        \\
        \lambda_i (u_i-h_i) &= 0 \,, &\qquad &\forall \, i= 1\,,\dots \,, N \,. \label{disc:four}
  \end{alignat}
In order to solve this forward contact problem in the discretized setting, we consider two approaches, an accelerated projected gradient method and a barrier method: 

\vspace{10pt}
\noindent
\textbf{(Accelerated) projected gradient method.} For the projected gradient method, we first consider the proximal gradient method \eqref{projgrad} on the discrete level, i.e.,
    \begin{equation}\label{prox_method_discr}
        \uf_n = P_{K_h} ( \uf_n - \tau (A \uf_n - \ff))\,,
    \end{equation}
where $P_{K_h}$ now denotes the projector onto $u_i\geq h_i$, which on the coefficient level reads 
    \begin{equation*}
        (P_{K_h} \yf)_i = h_i + \max\Kl{ y_i - h_i,0 } \,,
        \qquad \text{for} \,\, \yf = (y_i)_{i=1}^N \,.
    \end{equation*} 
To speed up the iteration, we use Nesterov acceleration \cite{Nesterov_1983}, which in this case yields 
    \begin{align}\label{apg}
    \begin{split}
        \uf_n &= P_{K_h} ( \zf_n - \tau (A \zf_n - \ff)) \,, 
        \\ 
        \zf_n &= \uf_n + \frac{t_{n-1} -1}{t_{n}} (\uf_n - \uf_{n-1}) \,,  
    \end{split}
    \end{align}
where $t_{n+1} := \frac{1}{2}(1 + \sqrt{1 + 4 t_n^2})$ and $\tau$ is a stepsize parameter which has to be chosen smaller than 
$\frac{2}{\norm{A}_{2}}$. This method works well for the test problems presented in Section~\ref{sect_numerics}.

\vspace{10pt}
\noindent
\textbf{Barrier method.} As an alternative to the projected gradient method, we also consider a barrier method for solving the discretized constrained optimization problem. The reason for a second method is that the mapping $a \to \uf$ is presumably not continuously differentiable with respect to the parameter $a$, as can be seen from the $\max$-function appearing in the projected gradient method \eqref{prox_method_discr}, and by our remarks on semismoothness above. With regard to the inverse problem, for which differentiability is strongly beneficial, we therefore consider a smooth approximation of the constrained forward problem using a log-barrier function. For this, we define a barrier function (for the constraint $u_i \geq h_i$) as follows:
    \begin{equation*} 
        g(\uf) := 
        \begin{cases} 
            -\log(u_i - h_i + \theta)\,,  & u_i \geq h_i\,, 
            \\
            \infty \,, & \text{else} \,,  
        \end{cases}
    \end{equation*}        
where $\theta>0$ is a safeguard parameter ensuring a finite value of the logarithm. The approximate solution of the constrained forward problem is then found by minimizing
    \begin{equation}\label{barr}
        J_b(u) =  \frac{1}{2} \uf^T A \uf - \ff^T \uf + \mu g(\uf) \,, 
        \qquad \text{s.t.\ } \quad 
        u_i \geq h_i \,, \quad \forall \, i = 1\,,\dots\,,N \,,
    \end{equation}
where $\mu$ denotes a penalty parameter which is updated during the iteration, and $A$ is the stiffness matrix including the given homogeneous Dirichlet boundary conditions. In order to solve this minimization problem, we use the projected Newton-type method 
    \begin{equation}\label{proj_Newton}
        \uf_{n+1} = P_{K_h}\kl{ \uf_n - \tau \nabla^2 J_b(\uf_n)^{-1} \nabla J(\uf_n) } \,,
    \end{equation}  
with a fixed stepsize $\tau$ and  iteratively updated $\mu$ and $\theta$, which are chosen as
    \begin{equation*}
        \mu_{n+1} = 0.9 \cdot \mu_n \,,
        \qquad \text{and} \qquad 
        \theta_{n+1} = \max\Kl{ 0.9 \cdot \theta_{n},10^{-12}} \,.
    \end{equation*} 
Note that the constraint $u_i \geq h_i$ in \eqref{barr} ensures that $g(\uf)$ is smooth on the admissible set (of $\uf$), and consequently $J_b(\uf)$ is (twice) differentiable and thus \eqref{proj_Newton} is well-defined.

\subsection{The discretized inverse contact problem} 

To solve the discretized inverse contact problem, we consider regularization schemes as discussed in Section~\ref{sect_regularization}. Since the forward operator $F$ is nonlinear, iterative optimization methods are typically used to minimize the corresponding Tikhonov functional \eqref{Tikh_cont}. However, most of these methods require $F$ to be differentiable with respect to $a$, which is not necessarily the case for the contact problem. Hence, for solving the inverse problem, we slightly modify the forward problem by considering the differentiable barrier problem formulation \eqref{barr} instead. In a sense, we approximate the forward operator $F$ by a differentiable operator $\tilde{F}$, for which we then use standard Tikhonov regularization. This approximation can be justified by noting that regularization theory allows for small perturbations of the operator $F$ without impairing the convergence analysis. 

Hence, in place of the true forward operator $F$, we consider the approximate mapping $\tilde{F}(a) := B \tilde{\uf}$, where $\tilde{\uf}$ denotes the solution (minimizer) of the barrier problem \eqref{barr} in the discretized FEM setting. It is not difficult to verify that $\tilde{\uf} = \tilde{\uf}(a)$ satisfies 
    \begin{equation*} 
        A \tilde{\uf} - \ff  + \mu \frac{1}{ \tilde{\uf} - \hf +\theta} = 0 \,. 
    \end{equation*} 
Here and below, the vector $\frac{1}{ \tilde{\uf} - \hf +\theta}$ stands for the vector $(\frac{1}{ \tilde{\uf}_i - \hf_i +\theta})_{i=1}^n$, i.e., where the inverse is taken component-wise. Furthermore, with a slight abuse of notation, we also denote by $\frac{1}{ \tilde{\uf} - \hf +\theta}$ and similar expressions the diagonal matrix with entries $(\frac{1}{\tilde{\uf} - \hf +\theta})_i$, $i=1,\ldots,n$, representing the discrete version of the multiplication operator        
    \begin{equation*}
        v(x) \to \kl{ \frac{1}{ \tilde{\uf}(x) - \hf(x) +\theta}}  v(x) \,.
    \end{equation*}
Next, note that the parameter $a$ only appears inside the stiffness matrix $A = A(a)$, and only as a linear factor. Hence, using the above notation, the Fr\'echet derivative 
$\tilde{\uf}'(a)b$ into direction $b$ can be calculated as the (unique) solution $\tilde{\uf}_{a;b}$ of
    \begin{equation*} 
        A(a) \tilde{\uf}_{a;b} + A(b) \tilde{\uf} + \mu \frac{1}{ (\tilde{\uf} - \hf +\theta)^2 }
        \tilde{\uf}_{a;b} = 0  \,,
    \end{equation*}
and thus
    \begin{equation*} 
        \tilde{\uf}'(a)b = \kl{ A(a) + \mu \frac{1}{ (\tilde{\uf} - \hf +\theta)^2 }}^{-1} A(b) \tilde{\uf} \,.
    \end{equation*} 
Furthermore, its $L_2$-adjoint is given by 
    \begin{equation*} 
        \tilde{\uf}'(a)^*\vf 
        =
        \kl{ \sum_{n=1}^N  \spr{\nabla \phi_k,\nabla \phi_n} (\tilde{\uf})_k  (\wf)_n }_{k=1}^N \,, 
    \end{equation*}
where
    \begin{equation*}    
        \wf := \kl{ A(a) + \mu \frac{1}{ (\tilde{\uf} - \hf +\theta)^2 }}^{-1} \vf \,,
    \end{equation*}
for the derivation of which we used that $(A(a) + \mu \frac{1}{ (\uf - \hf +\theta)^2 })$ is a selfadjoint matrix. 

The Fr\'echet derivative $\tilde{F}'(a)$ of the full forward operator $\tilde{F}$ (barrier formulation) including the observation operators (now matrices) $B$ is then simply given by
    \begin{equation*}
        \tilde{F}'(a)b = B\, \tilde{\uf}'(a)b \,,
    \end{equation*}
and the $L_2$-adjoint $\tilde{F}'(a)^*$ now reads
    \begin{equation*}
        \tilde{F}'(a)^* y  = \tilde{\uf}'(a)^*(B^*y) \,.
    \end{equation*}

\paragraph{The iterative regularization scheme}

Having derived all the necessary ingredients above, we can now consider Landweber iteration as our iterative regularization method, which approximates the ground truth coefficient $a^\dagger$ iteratively via ($\tau$ is a stepsize)
    \begin{equation}\label{modland}
        a_{k+1} = a_k - \tau \tilde{F}'(a_k)^*(F(a_k) - y) \,.
    \end{equation}
Note that different Hilbert space regularization can be introduced by considering the adjoint $\tilde{F}'(a_k)^*$ between appropriate spaces \cite{Simonadjoint}. Furthermore, in order to accelerate the Landweber iteration \eqref{modland}, we propose to use its Nesterov-accelerated version~\cite{Hubmer_Ramlau_2017}: 
   \begin{align} \label{Nest}
   \begin{split}
        b_n = a_n + \frac{n-1}{n+2} \kl{a_n - a_{n-1}} \,,
        \\
        a_{n+1}= a_n - \tau {\tilde{F'}(b_n)}^*(F(b_n) - y) \,.
        \end{split}
    \end{align}
This is the method which was implemented for obtaining our numerical results below. 

Both methods, \eqref{modland} and \eqref{Nest}, can be related to Tikhonov regularization by replacing the above least-squares gradient by the gradient of the Tikhonov functional. The results are comparable; the methodological difference is that the iterative schemes have to be terminated by a stopping criterion, e.g., the discrepancy principle, while for the Tikhonov-gradient variant, the regularization parameter has to be selected by a suitable parameter choice rule, and the iteration converges (if certain conditions are satisfied) to a stationary point of the Tikhonov functional. 

Concerning the convergence properties of the proposed methods \eqref{modland} or \eqref{Nest}, we here only sketch how a convergence analysis may be conducted: the main difference between \eqref{modland} and \eqref{Nest} and established methods is the use of a ``wrong'' derivative and its adjoint. However, such a possibility already exists in the seminal article on Landweber iteration for nonlinear inverse problems \cite{Hanke_Neubauer_Scherzer_1995}. There, the main condition which needs to be verified for a successful analysis of \eqref{modland} is the tangential cone-type condition 
    \begin{equation*}
        \norm{ F(a) - F(\tilde{a}) - \tilde{F}'(\tilde{a})(a-\tilde{a}) } \leq \eta \norm{ F(a) - F(\tilde{a}) } \,,
        \qquad
        \eta < \frac{1}{2} \,.
    \end{equation*}
We leave an analysis of this condition for the contact problem to future research.

\section{Numerical experiments}\label{sect_numerics}

In this section, we present a number of numerical experiments illustrating the solution of the forward and inverse contact  problems with the methods introduced above. For all experiments, we use a standard ($P_1$-elements) finite-element discretization on a uniform grid over the unit square $[0,1]^2$ with $n$ grid points in each axial direction. The coefficient $a$ was also discretized by $P_1$-finite elements, and so was the indenter function $h$. All of the experiments were conducted in Matlab (R2025a) on a standard desktop computer
(Intel Core i7-14700 with 31 GiB RAM, Linux openSUSE Leap 15.6).

\subsection{Forward problem verification}

In our first series of tests, we verify the three algorithms for the forward contact problem described above, i.e., the projected gradient method (PG), the Nesterov-accelerated projected gradient method (NPG), and the barrier method (BM). In particular, we investigate how well and how fast these algorithms can approximate an exact solution. 

As an error indicator, we use both the maximum-norm and the (discrete) $\ell_2$-norms between the computed solution $u_{\text{comp}}$ and an exact solution $u_{\text{ex}}$, i.e., 
    \begin{equation*}
        \text{err}_\infty := \norm{u_{\text{comp}}-u_{\text{ex}}}_\infty\,,
        \qquad \text{and} \qquad 
        \text{err}_2 := \norm{u_{\text{comp}}-u_{\text{ex}}}_{\ell_2} \,.
    \end{equation*} 
Here, $u_{\text{comp}}$ denotes the solution obtained by one of the above algorithms, while the reference ``exact'' solution $u_{\text{ex}}$ is computed on a refined grid ($n = 500$). Furthermore, as an optimality indicator, we use the residual norm of the KKT system, i.e., 
    \begin{equation*}
        \text{err}_{\text{KKT}} := \norm{ \begin{pmatrix} \min(A u -f,0) \\  \min(u-h,0) \\ \norm{ (Au -f)(u-h) }_\infty \end{pmatrix} }_{\ell_2} \,. 
    \end{equation*}
Here, the first expression inside the norm measures the violation of $\lambda \geq 0$, the second one that of $u \geq h$, and the last one the violation of the KKT conditions. These expressions are computed nodalwise, and the Euclidean norm of all expressions is used above. 

For our first test (TestCase~1), we consider the Laplace equation, i.e., $a = 1$ and $f = 0$, with an indenter which is the indicator function on the square $[0.35,0.65]^2$. The iteration is terminated as soon as $\text{err}_{\text{KKT}}$ falls below $10^{-8}$ or a maximum number of iterations ($5000$) is reached. The only sensitive parameter for our gradient methods is the stepsize $\tau$, which should be chosen such that $\tau \norm{A}_{2}< 2$; here we use the more conservative choice $\tau = 0.1/\norm{A}_{2}$. For the barrier method, we use the stepsize $\tau = 0.1$.

\begin{figure}[ht!]
    \includegraphics[width=0.45\textwidth]{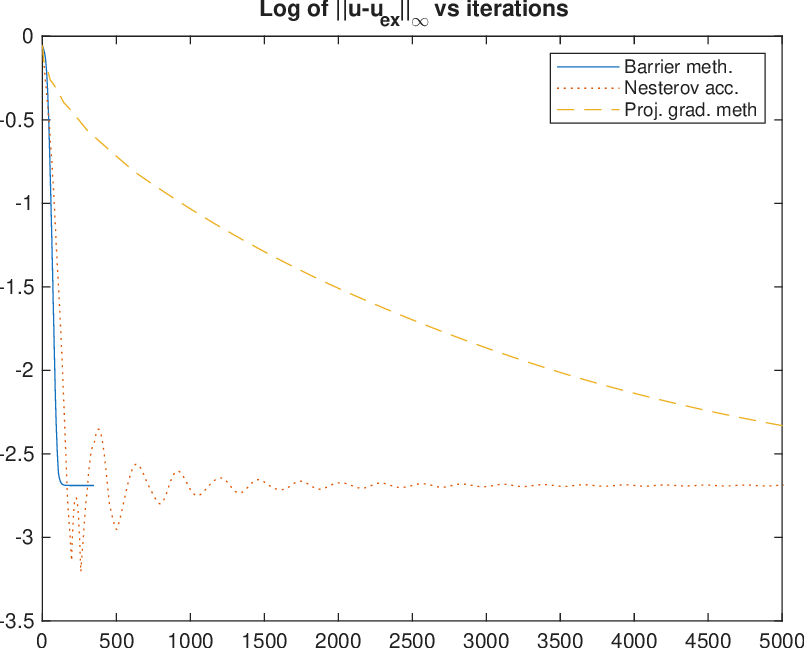}
    \includegraphics[width=0.45\textwidth]{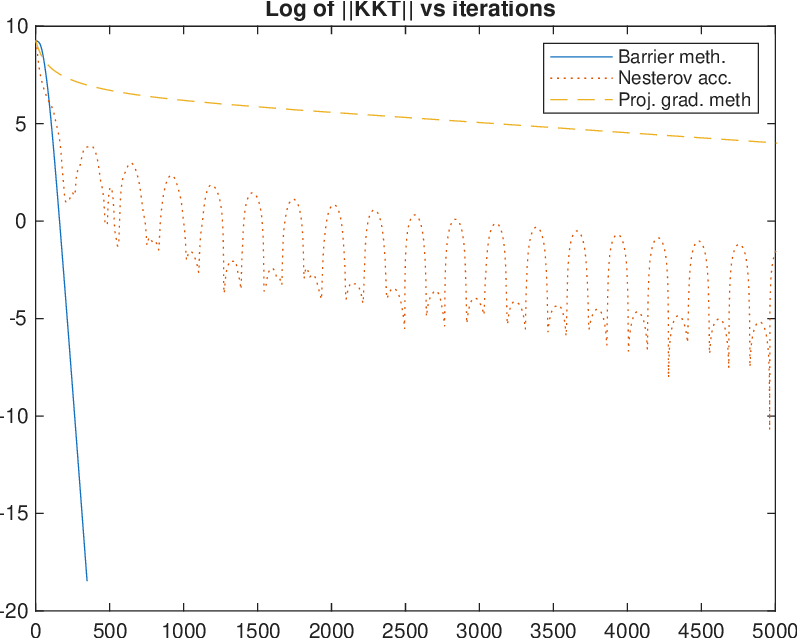}
        \caption{Log of $\text{err}_\infty$ (left) and log of $\text{err}_{\text{KKT}}$ (right) vs.\ the number of iterations for TestCase~1 and a discretization of $n=40$, for the barrier method (blue), Nesterov-accelerated projected gradient method (red, dotted), and projected gradient method (yellow, dashed).}
\label{fig1}
\end{figure}

\begin{table}[ht!]
\begin{center}
\caption{Error indicators for various discretizations, methods, and TestCase~1 and 2.}
\label{tab1}
    \begin{tabular}{ccccc}
        &  $\log(\text{err}_\infty)$  & $\log(\text{err}_{L_2})$  &
        $\begin{array}{c} \text{\# iter. for} \\
        \text{err}_{\text{KKT}} <10^{-8} \end{array}$
        & CPU time(s)\\ \hline
        \multicolumn{4}{l}{\textbf{TestCase~1}} \\ \hline 
        $n=20$  & & &  \\ \hline 
        \multicolumn{1}{l|}{BM} & -2.35  &   -3.2988 & 323
        & 0.13\\ 
        \multicolumn{1}{l|}{NPG} & -2.35   &  -3.2988 & 3692  & 1.03\\ 
        \multicolumn{1}{l|}{PG} &  -2.35 &  -3.2988&  4834 
        & 0.69 \\ \hline 
        $n=40$  & & &  \\ \hline 
        \multicolumn{1}{l|}{BM} &-2.6893   &    -3.6913     & 347 & 0.37 \\ 
        \multicolumn{1}{l|}{NPG} &-2.6873 &    -3.6885  & $\ge $5000 & 1.40\\ 
        \multicolumn{1}{l|}{PG} &  -2.3328 &   -3.1552 &  $\ge $5000 & 0.80\\ \hline 
        $n=80$  & & &  \\ \hline 
        \multicolumn{1}{l|}{BM} &-2.8549    & -3.9528  & 347 & 1.72\\ 
        \multicolumn{1}{l|}{NPG} & -2.8875   &   -4.0151   & $\ge $5000 & 1.84\\ 
        \multicolumn{1}{l|}{PG} &  -0.7442 &   -1.4140 &  $\ge $5000 & 1.26 \\ \hline 
        $n=160$  & & &  \\ \hline 
        \multicolumn{1}{l|}{BM} &-2.9466     &   -4.0876   & 394 & 7.00 \\ 
        \multicolumn{1}{l|}{NPG} &  -3.3405     &    -4.2873   & $\ge $5000 & 3.51\\ 
        \multicolumn{1}{l|}{PG} &  -0.2463 &    -1.0694 &  $\ge $5000 & 2.96 \\ \hline 
        \multicolumn{4}{l}{\textbf{TestCase~2}} \\    \hline 
        $n=20$  & & &  \\ \hline 
        \multicolumn{1}{l|}{BM} & -2.3029  &   -4.2973  & 324 & 0.12\\ 
        \multicolumn{1}{l|}{NPG} & -2.3029    &  -4.2973  & 3667 & 0.63\\ 
        \multicolumn{1}{l|}{PG} &  -2.3029  &  -4.2973 &  4645 & 0.45\\
        \hline 
        $n=40$  & & &  \\ \hline 
        \multicolumn{1}{l|}{BM} & -2.8150    &    -4.5637      &  348 & 0.36\\ 
        \multicolumn{1}{l|}{NPG} & -2.8150 &    -4.5637   & $\ge $5000 & 0.95  \\ 
        \multicolumn{1}{l|}{PG} &   -2.8150  &   --4.5637  &  $\ge $  5000 & 0.58 \\ \hline 
        $n=80$  & & &  \\ \hline 
        \multicolumn{1}{l|}{BM} &
        -3.2752    & -5.2024  & 372 & 1.74 \\ 
        \multicolumn{1}{l|}{NPG} & -3.2752    &    -5.1872    & $\ge $5000 & 1.45\\ 
        \multicolumn{1}{l|}{PG} &  -1.2072 &    -2.1834 &  $\ge $5000 & 1.03\\ \hline 
        $n=160$  & & &  \\ \hline 
        \multicolumn{1}{l|}{BM} & -3.9207     &    -6.1688   & 395 & 7.40\\ 
        \multicolumn{1}{l|}{NPG} &  -3.9047     &     -5.2068   & $\ge $5000 & 3.19\\ 
        \multicolumn{1}{l|}{PG} &  -0.6389 &     -1.7066 &  $\ge $5000 & 2.74 \\ \hline 
    \end{tabular}
\end{center}
\end{table}

The results of this case are summarized in Table~\ref{tab1} and Figure~\ref{fig1}. In Table~\ref{tab1}, we summarize various performance indicators (errors and wall-clock
CPU time in seconds), while Figure~\ref{fig1} depicts the error curve for the maximum norm ($\text{err}_\infty$) and the KKT residual ($\text{err}_\text{KKT}$) over the number of iterations. From these results, we first observe the relative slowness of the projected gradient method, which as expected is outperformed by the other two methods. The barrier method, being a second-order method, requires the least number of iterations, exhibiting a steep decrease of the KKT residual during the iteration. Furthermore, we see that the NPG method outperforms the PG method: the error is quite small and the solution acceptable even after a moderate ($<1000$) number of iterations, after which the error remains almost constant; the method does not terminate because the KKT residual is decreasing slower than for the barrier method. 

For our second test (TestCase~2), we perform the same experiment as above but with a different set of parameters. More precisely, for the indenter $h$ we take the indicator function of a circle with center $(0.5,0.5)$ and radius $0.25$, and ``cut off'' its top by the linear function $g(x_1,x_2)= x_1$; see Figure~\ref{fig2}(left). Furthermore, we choose
    \begin{equation*}
    \begin{split}
        f(x_1,x_2) &= (-10)\sin(\pi x_1)\sin(2\pi x_2) \,,
        \\
        a(x_1,x_2) &= 1 + \tfrac{1}{2}\exp(-20((x_1-0.4)^2 + (x_2-0.5)^2)) \sin(2\pi x_1) \,,
    \end{split}
    \end{equation*}
Both the indenter and the corresponding solution are illustrated in Figure~\ref{fig2}. The results are quite similar to the first experiment; the characteristics are given in Table~\ref{tab1}.

\begin{figure}
    \includegraphics[width=0.45\textwidth,height=0.405\textwidth]{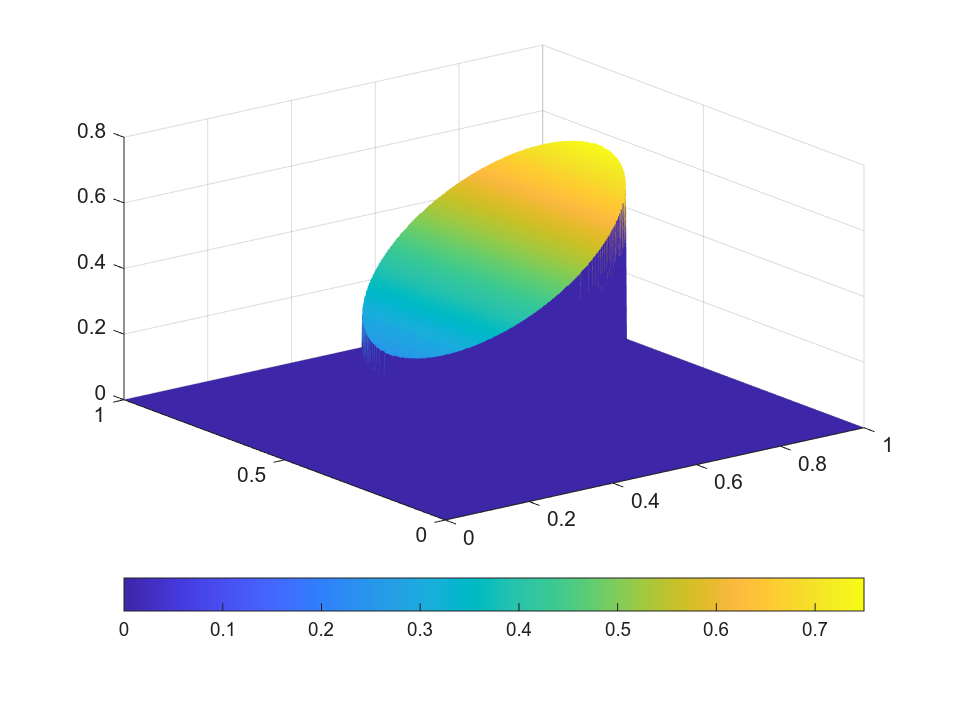}
    \includegraphics[width=0.45\textwidth,height=0.405\textwidth]{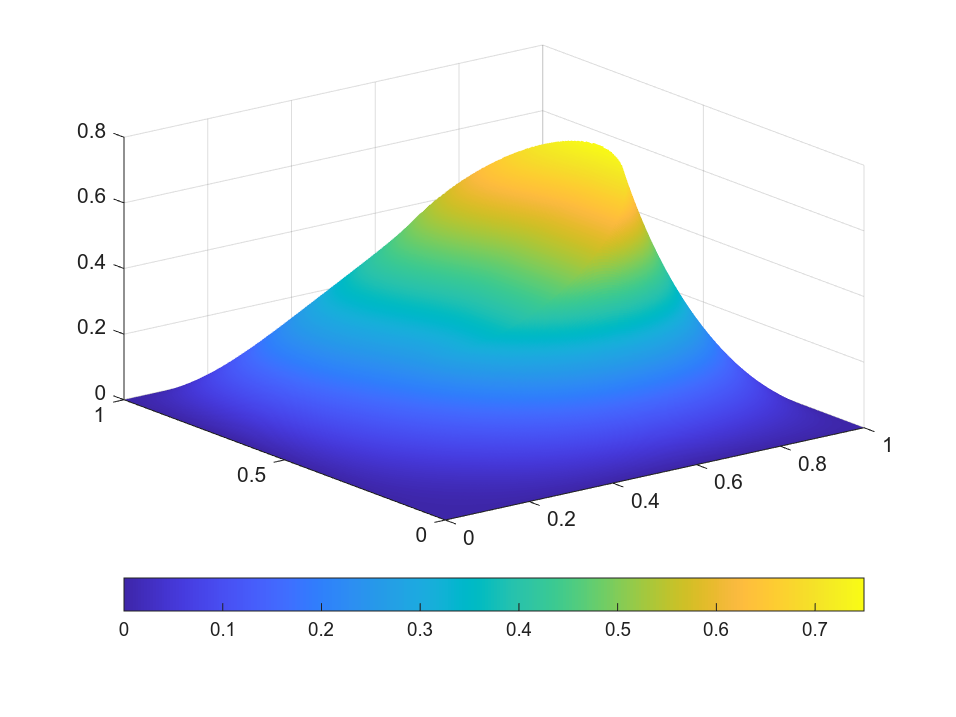}
\caption{Indenter $h$ (left) and solution (right) for TestCase~2.} \label{fig2}
\end{figure}

\subsection{Numerical results for the inverse problem}\label{sec6.2}

In our second series of tests, we consider the solution of the inverse contact problem via the iterative regularization procedure \eqref{Nest} and illustrate some of its features.

\vspace{10pt}
\noindent
\textbf{Experiment 1.}
In our first experiment, we consider the identification of the coefficient $a$ from full measurements of $u$ with different indenters and noise levels. As before, we use a fixed grid on the unit square, now with $n = 50$ nodes in each direction. For the right-hand side, we select $f = 0$, and for the ground truth parameter $a^\dagger$, we choose 
    \begin{equation*}
        a^\dagger(x_1,x_2) := 1 + \frac{1}{2} e^{-20(x_1-0.4)^2 +(x_2-0.5)^2} \sin(2\pi x_1) \,.
    \end{equation*}
We assume that $a^\dagger$ is known on the boundary $\partial \Omega$, and use $H_0^1(\Omega)$-regularization for $a$, entering via the adjoint ${\tilde{F'}(\cdot)}^*$ in \eqref{Nest}. Note that our chosen ground truth $a^\dagger$ changes its sign around the reference value $1$ to avoid simplifications due to monotonicity effects, which are known to act beneficially in parameter identification problems \cite{KindermannMono,Harrach}. 

In this experiment, we vary the noise level and the indenter size to illustrate non-identifiability at the contact shape in the worst case, and feasibility of the reconstruction method and the inverse contact problem in the best case. That is, we choose three different indenters as indicator functions of a circle centered at $(0.5,0.5)$ with radii 
    \begin{equation*}
        \text{indenterradius} \in \Kl{0.1, 0.25,0.5} \,.
    \end{equation*}
Furthermore, we vary the relative noise level with 
    \begin{equation*}
        \delta_{rel} \in \Kl{0.1\%,1 \%,10\%} \,.
    \end{equation*}
To avoid inverse crimes, the noisy data was obtained by first calculating a fine-grid exact solution $u$ and then projecting it onto the calculation grid ($n = 50$). Then, uniformly distributed random noise was added to achieve the desired noise level. In this way, the noise in our data consists of two components: discretization errors and additive noise.

\begin{figure}[p]
\begin{center} 
    \begin{minipage}{0.49\textwidth}
        \centering
        \includegraphics[width=\textwidth]{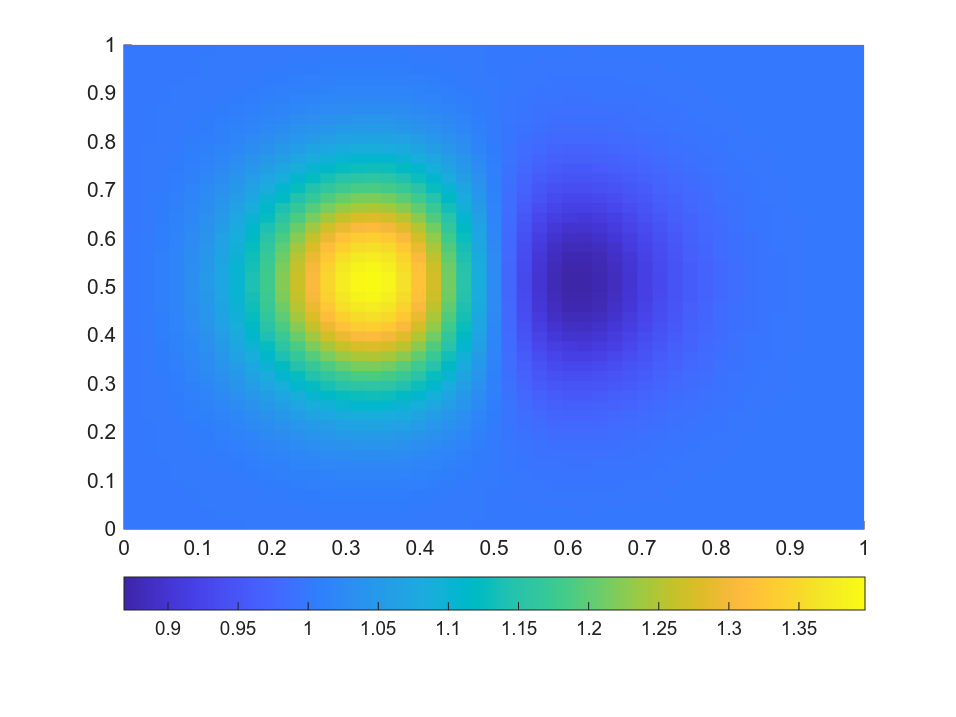}\\
        \footnotesize ground truth  $a^\dagger$
    \end{minipage} \hspace*{-5ex}
    \begin{minipage}{0.49\textwidth}
        \centering
        \includegraphics[width=0.75\textwidth]{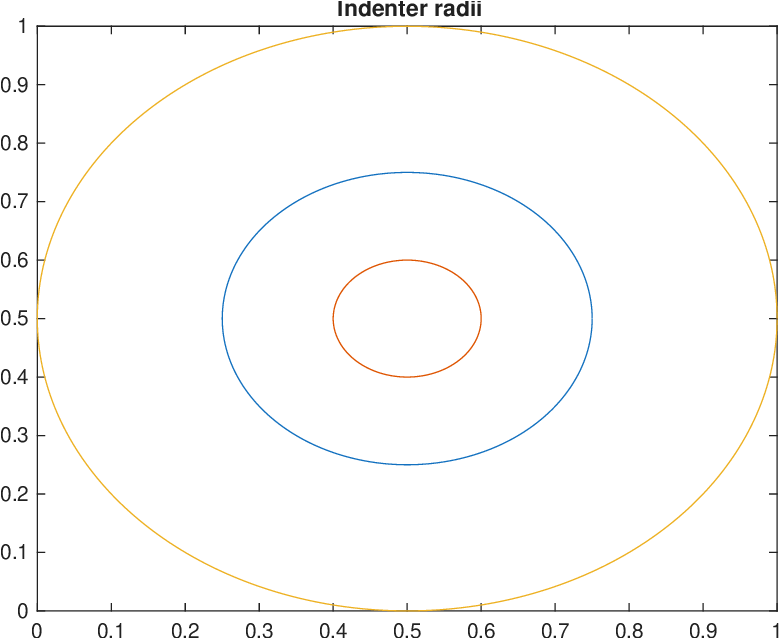}\\
        \footnotesize Size and shape of 3 indenters  
    \end{minipage}
\end{center}
\begin{center}
    \footnotesize Reconstructed $a$ for indenterradius = $0.1$
\end{center}
\vspace{-2ex}
    \begin{minipage}{0.31\textwidth}
        \centering
        \includegraphics[width=\textwidth]{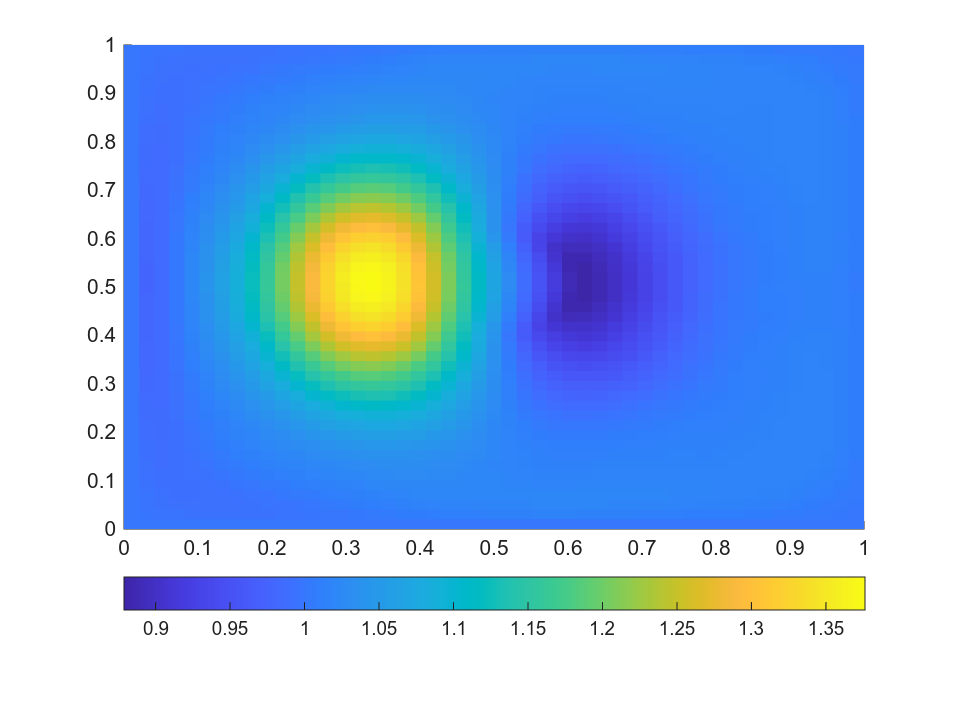}\\
        \footnotesize $\delta_{rel} = 0.1\%$
    \end{minipage}
    \begin{minipage}{0.31\textwidth}
        \centering
        \includegraphics[width=\textwidth]{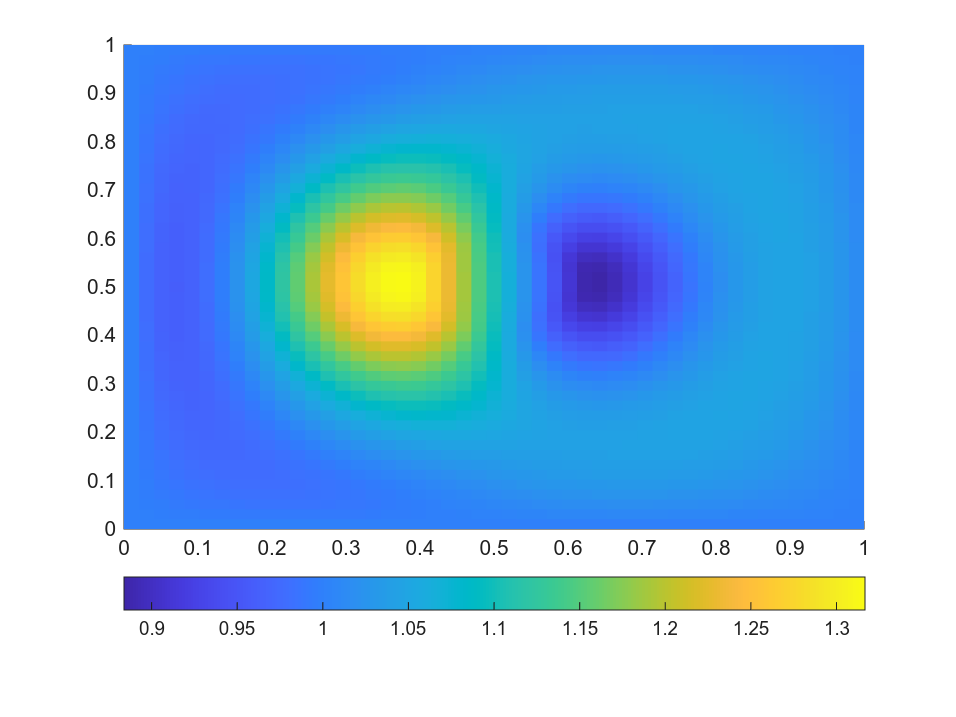}
        \\
        \footnotesize$\delta_{rel} = 1\%$
    \end{minipage}
    \begin{minipage}{0.31\textwidth}
        \centering
        \includegraphics[width=\textwidth]{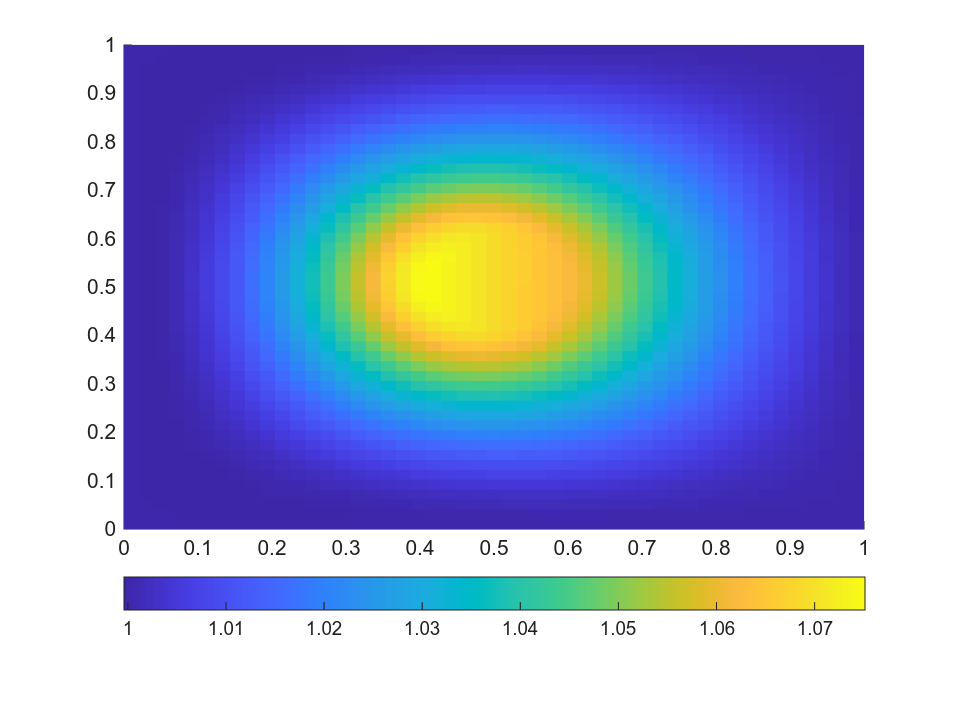}
        \footnotesize $\delta_{rel} = 10\%$
    \end{minipage}
\begin{center}
    \footnotesize 
    Reconstructed $a$ for  indenterradius = $0.25$
\end{center}
\vspace{-2ex}
    \begin{minipage}{0.31\textwidth}
        \centering
        \includegraphics[width=\textwidth]{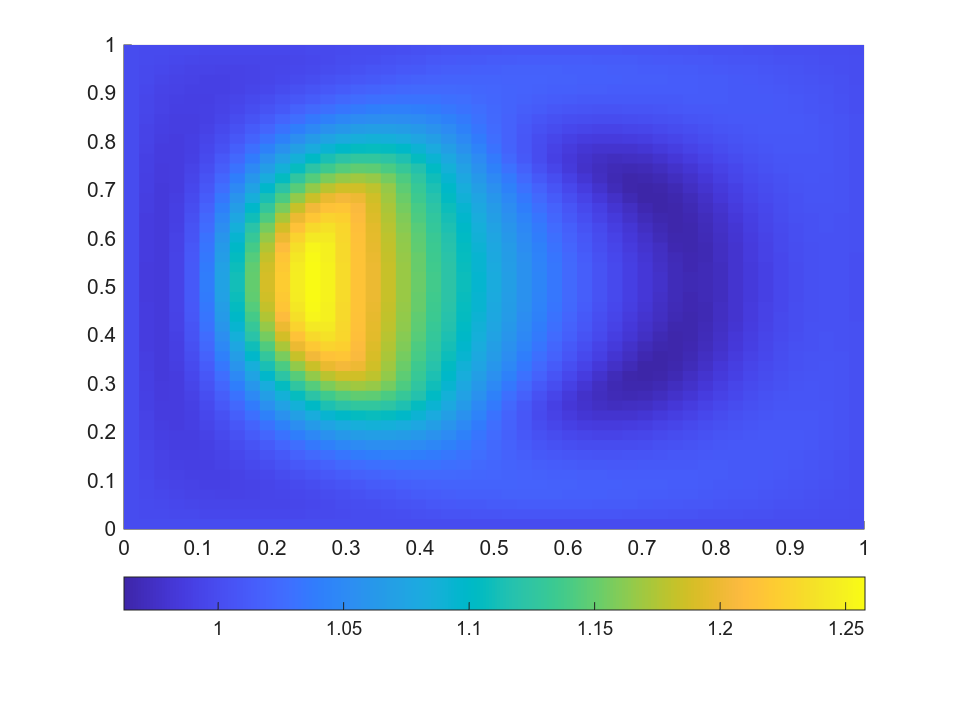}\\
        \footnotesize $\delta_{rel} = 0.1\%$
    \end{minipage}
    \begin{minipage}{0.31\textwidth}
        \centering
        \includegraphics[width=\textwidth]{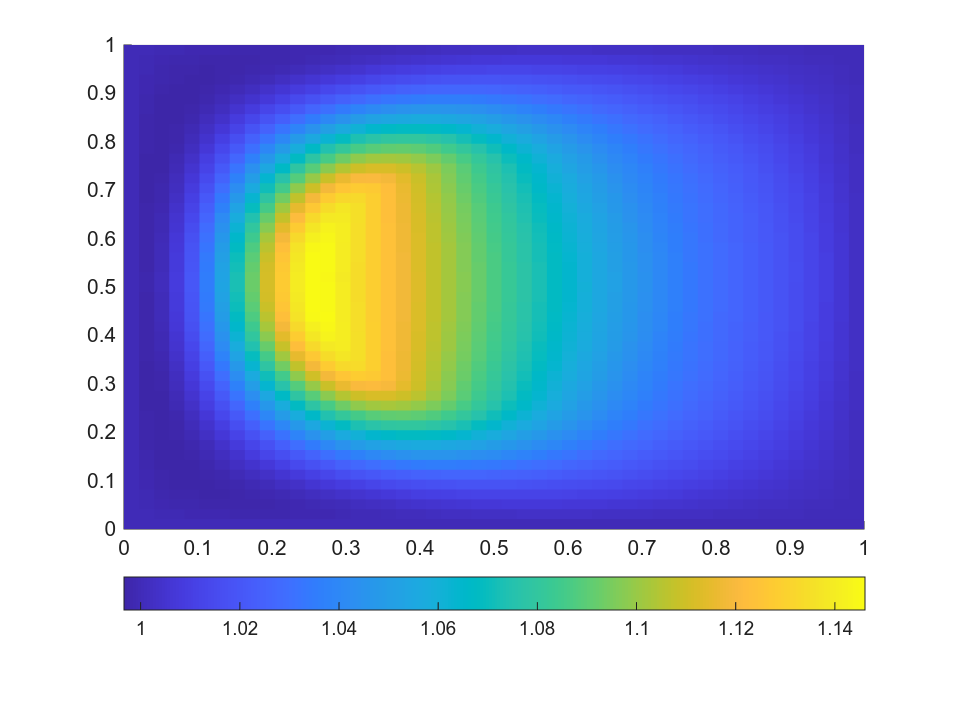}
        \\
        \footnotesize$\delta_{rel} = 1\%$
    \end{minipage}
    \begin{minipage}{0.31\textwidth}
        \centering
        \includegraphics[width=\textwidth]{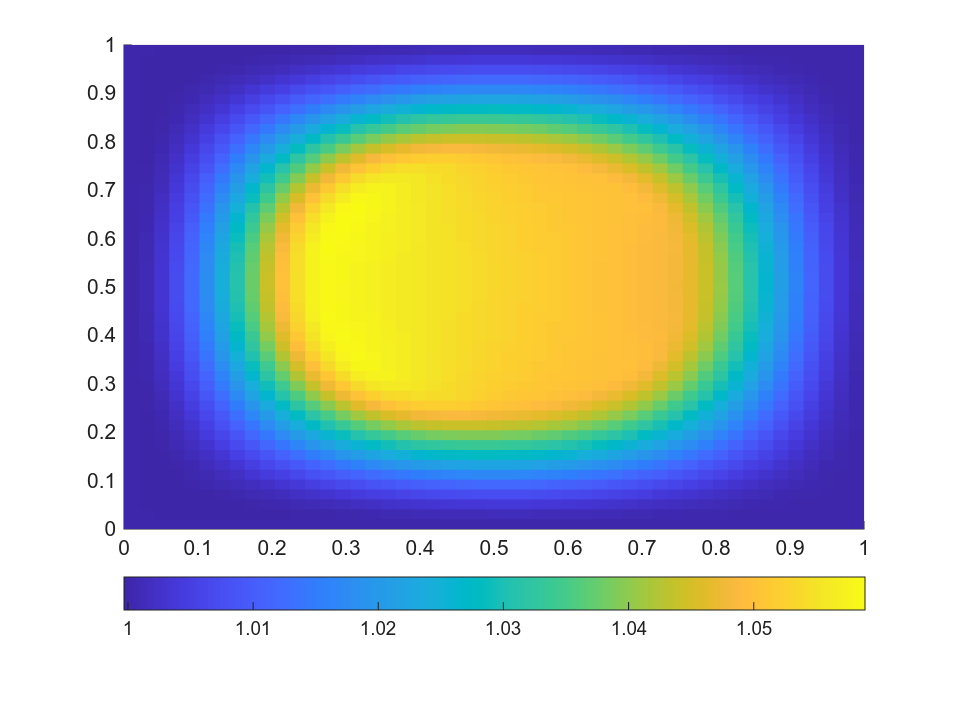}
        \footnotesize $\delta_{rel} = 10\%$
    \end{minipage}
\begin{center}
    \footnotesize 
    Reconstructed $a$ for indenterradius = $0.5$
\end{center}
\vspace{-2ex}
    \begin{minipage}{0.31\textwidth}
        \centering
        \includegraphics[width=\textwidth]{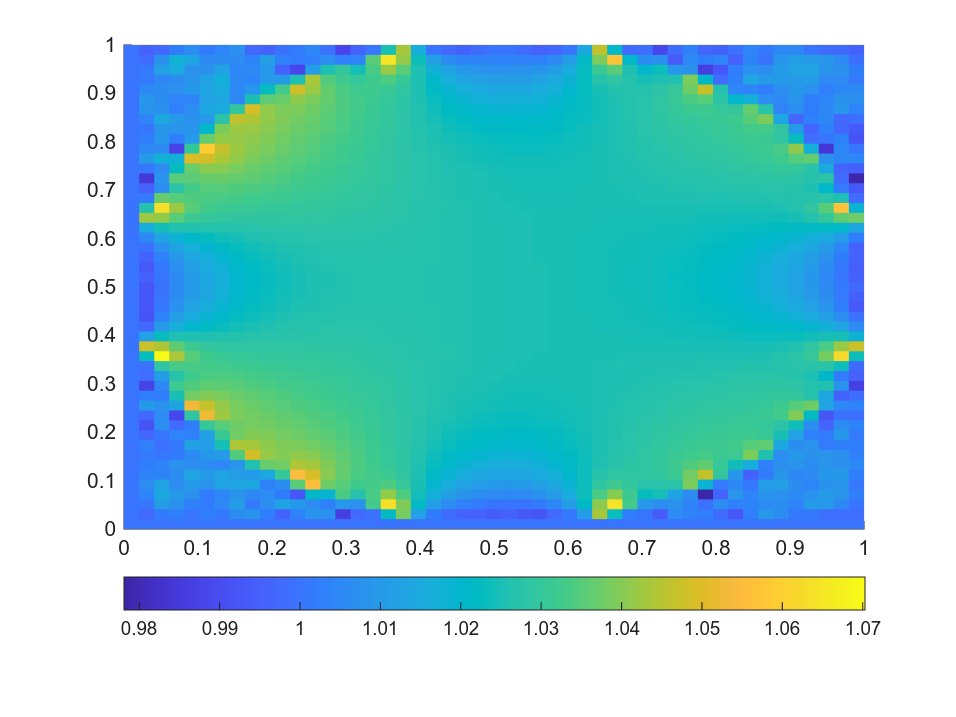}\\
        \footnotesize $\delta_{rel} = 0.1\%$
    \end{minipage}
    \begin{minipage}{0.31\textwidth}
        \centering
        \includegraphics[width=\textwidth]{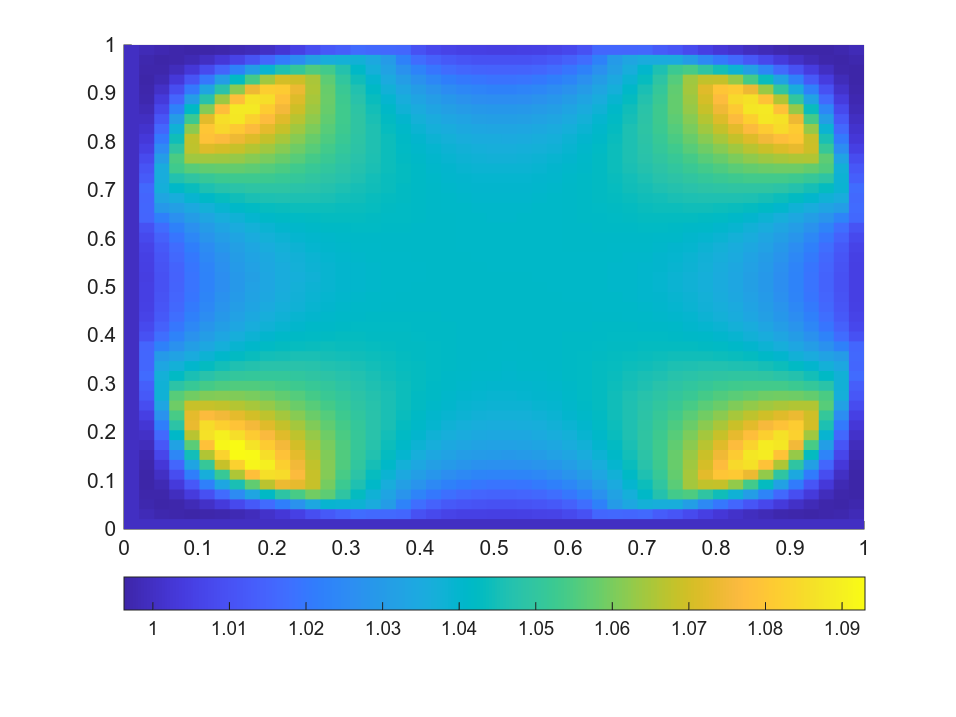}
        \\
        \footnotesize$\delta_{rel} = 1\%$
    \end{minipage}
    \begin{minipage}{0.31\textwidth}
        \centering
        \includegraphics[width=\textwidth]{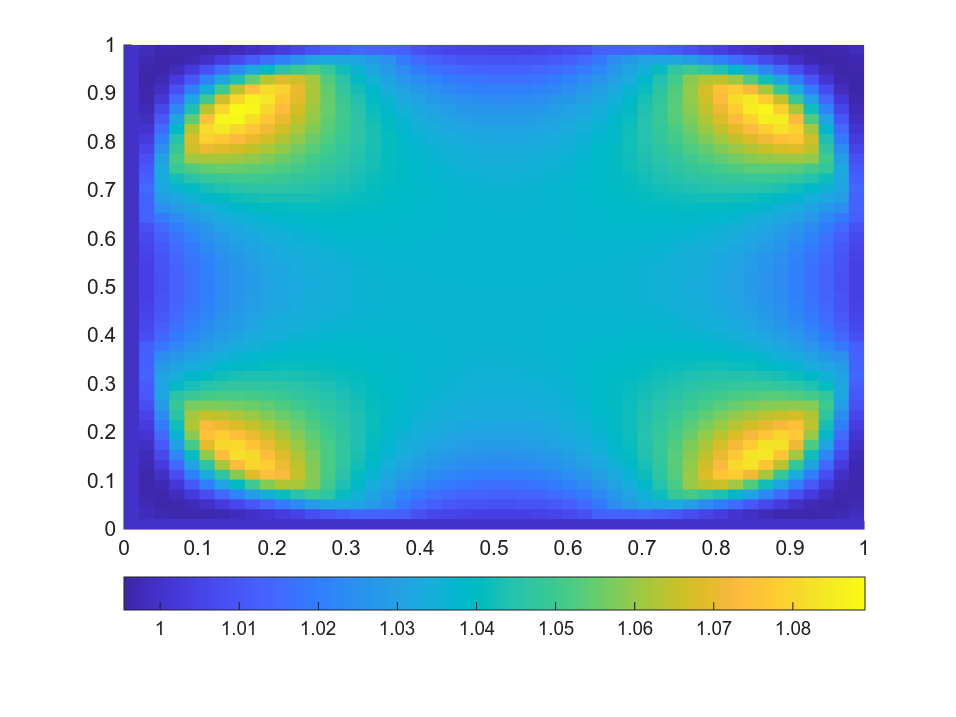}
        \footnotesize $\delta_{rel} = 10\%$
    \end{minipage}
\caption{Experiment 1: Top: ground truth and 3 different indenter shapes. Results: Left to right: increasing noiselevel: Top to bottom: increasing indenter size.}
\label{figex1}
\end{figure}

\newcommand{\sss}{\scriptsize}
\begin{table}[ht!]
\caption{Experiment~1 with the setup as in Figure~\ref{figex1} with different indenter radii: results using \eqref{Nest} and, in parentheses, when using Landweber iteration \eqref{modland}.}
\label{tabex1}
\begin{tabular}{l|ccccc}
    & 
    $\begin{array}{c}
    \text{\sss rel. err:}\\[-1ex]
    \text{\sss optimal $k$}\\
    \frac{\norm{a_{k_{opt}} - a^\dagger}_{H_0^1}}{\norm{a^\dagger-1}_{H_0^1}}
    \end{array}$
    &
    $\begin{array}{c}
    \text{\sss rel. err:}\\[-1ex]
    \text{\sss  discrep. princ.}\\
    \frac{\norm{a_{k_{disc}} - a^\dagger}_{H_0^1}}{\norm{ a^\dagger-1 }_{H_0^1}}
    \end{array}$
    &
    $\begin{array}{c}
    \text{\sss discrep. factor}\\[-1ex]
    \text{\sss optimal $k$}\\
    \frac{\norm{F(a_{k_{opt}}) - y^\delta}}{\norm{y-y^\delta}}
    \end{array}$
    & 
    $\begin{array}{c}
    \text{\sss Stop. index}\\[-1ex]
    \text{\sss  optimal}\\
    k_{opt}
    \end{array}$
    &
    $\begin{array}{c}
    \text{\sss Stop. index}\\[-1ex]
    \text{\sss discrepancy}\\
    k_{disc}
    \end{array}$ 
    \\ \hline 
    & \multicolumn{5}{c}{indenterradius = $0.1$}\\
    $ \delta=0.1\% $ & 0.164 (0.189) & 0.222 (0.252)& 0.33
    (0.56)& 381 (\textgreater 3000) & 76 (674) \\ 
    $ \delta=1\% $ &  0.457 (0.375)& 0.483 (0.500) 
    & 0.80 (0.49)& 21 (347) & 16 (49) \\ 
    $ \delta=10\% $ & 0.792 (0.773) & 0.793 (0.773) & 
    1.19 (0.97) & 1 (2) & 2 (2) \\ 
    \\ \hline 
    & \multicolumn{5}{c}{indenterradius = $0.25$}\\
    $ \delta=0.1\% $ & 0.305 (0.314) & 0.339 (0.331) & 2.10
    (0.48) & 24 (\textgreater 3000)& 47 (372) \\ 
    $ \delta=1\% $ &  0.615 (0.595) & 0.692 (0.617) &
    1.35 (1.32) & 7 (14) & 10 (23) \\ 
    $ \delta=10\% $ & 0.762 (0.761) & 1.326 (1.025) & 1.10
    (1.10) & 1 (1) & 3  (3) \\ 
    \\  \hline 
    & \multicolumn{5}{c}{indenterradius = $0.5$}\\
    $ \delta=0.1\% $ & 0.819 (0.878) & 0.982 (0.984)  & 
    0.58 (0.66) & 660 (\textgreater 3000) & 13 (30) \\ 
    $ \delta=1\% $ &  0.796 (0.793)  & 0.989 (0.991)   
    & 0.82 (0.81) & 26 (129) & 2 (2)  \\ 
    $ \delta=10\% $ & 0.801 (0.801) & 0.956 (0.956) & 0.98 
    (0.98)  & 5 (9) & 1 (1) \\ 
    \\ 
\end{tabular}
\end{table}

For each combination of indenterradius and noiselevel, we use the iterative method \eqref{Nest} as described above with at most $k_{max} = 3000$ iterations. In Figure~\ref{figex1}, we present the \emph{optimal} result which minimizes $\norm{a_k - a^\dagger}_{H_0^1(\Omega)}$ over the iteration index $k$. Furthermore, in Table~\ref{tabex1} we present two corresponding error measures, namely
    \begin{equation*}
        \text{relative error} := \frac{\norm{a_k - a^\dagger}_{H_0^1(\Omega)}}{\norm{a^\dagger-1}_{H_0^1(\Omega)}} \,,
        \qquad
        \text{discrepancy factor} := \frac{\norm{F(a_k) -y_\delta}}{\norm{y-y^\delta}} \,,
    \end{equation*}
both for the optimal case and when the stopping index is determined via the discrepancy principle with the discrepancy factor $1.01$. We also performed the same experiments using simple Landweber iteration \eqref{modland}, and the results are given in parentheses. 

The conclusions which can be drawn from these results is that the iterative regularization works well with the following expected results: if enough information (small indenter radius) is available, then we find good reconstructions for  moderate noiselevel ($<1\%$). In case too little information is available (e.g., for the large indenter, cf.~the non-identifiability in Theorem~\ref{th3.1}), then no reasonable reconstruction can be obtained; we again emphasize that this is not a flaw of the method but stems from the problem itself. Compared to simple Landweber iteration, the results are comparable, often better, but with the benefit of a substantial reduction in the required number of iterations. 

From Table~\ref{tabex1}, we also observe an unexpected result regarding the discrepancy principle for the optimal stopping rule: the optimal discrepancy index $\tau_{\text{opt}}$ is in many cases well below $1$, which at first contradicts classical theory which requires $\tau \geq 1$. We think that this is related to a non-dense range of the forward operator (non-attainable case): In the linear case, if the forward operator has a non-zero complement in the data space, one should project the residual and the noise onto the closure of the range and use as $\delta$ the effective noiselevel $\delta_{\text{eff}} = \|P_{\overline{R(A)}}(y-\yd)\|$, with $P_{\overline{R(A)}}$ the orthogonal projector onto the range, and a corresponding effective discrepancy principle (see, e.g. \cite[p.~84]{Engl_Hanke_Neubauer_1996}), otherwise a discrepancy factor below $1$ can appear. (This workaround is hardly doable in the nonlinear case, where the range depends on the iteration or on the exact solution.)

\begin{figure}[ht!]
\begin{minipage}{0.31\textwidth}
    \centering
    \includegraphics[width=\textwidth]{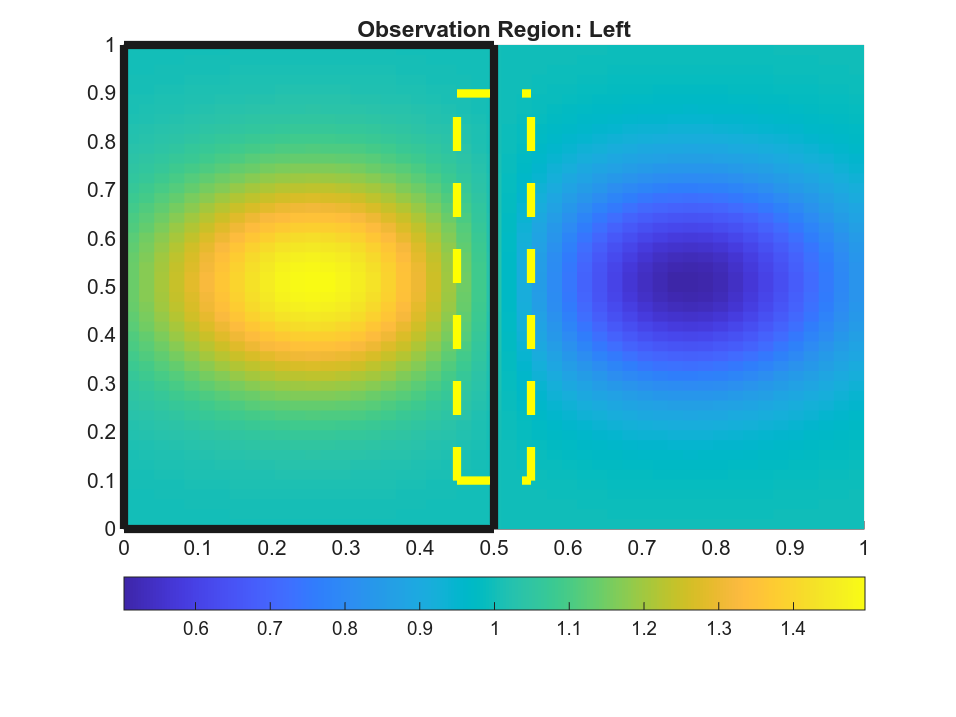}
    \footnotesize 
    Ground truth $a^\dagger$ \\
    Observation region: black\\
    Indenter: yellow dashed\\
    \includegraphics[width=\textwidth]{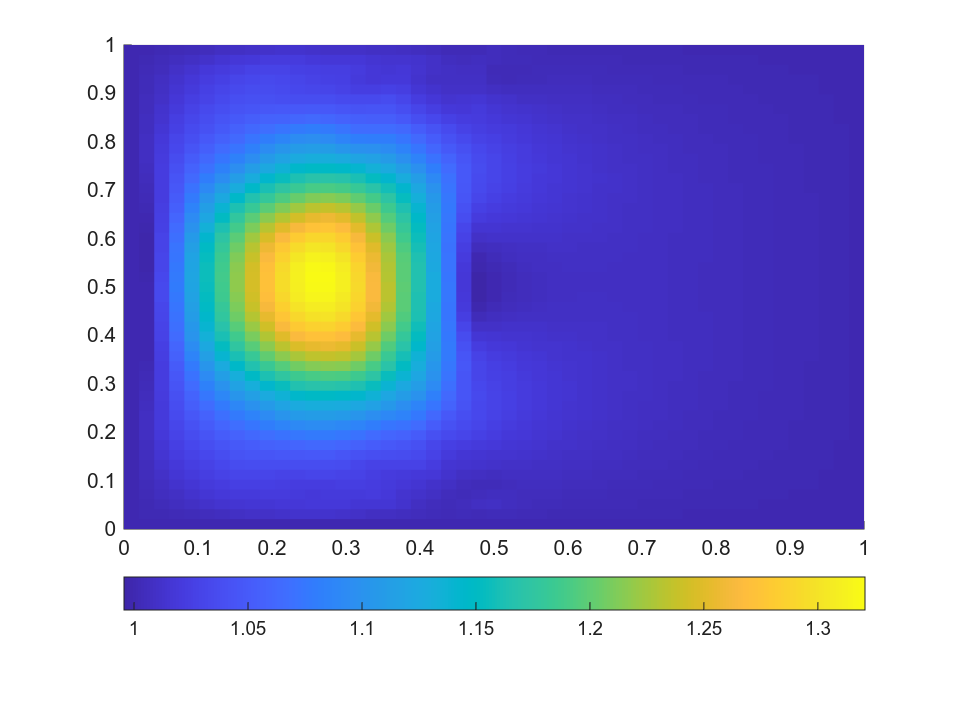}
    \footnotesize Reconstruction
\end{minipage}
\begin{minipage}{0.31\textwidth}
    \centering
    \includegraphics[width=\textwidth]{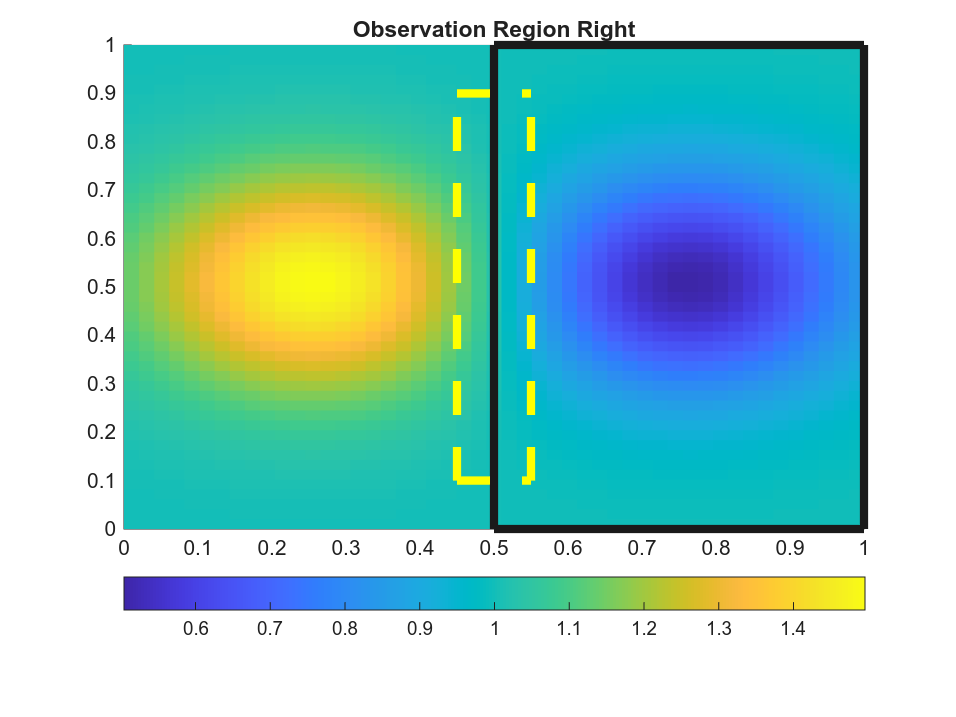}
    \footnotesize 
    Ground truth $a^\dagger$ \\
    Observation region: black\\
    Indenter: yellow dashed\\
    \includegraphics[width=\textwidth]{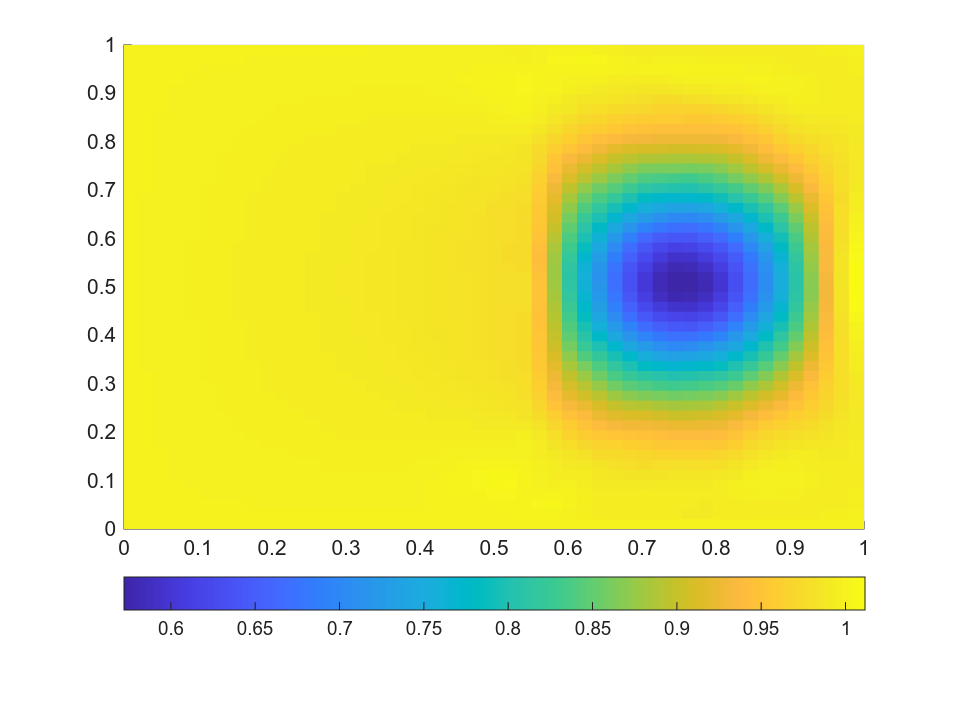}
    \footnotesize Reconstruction
\end{minipage}
\begin{minipage}{0.31\textwidth}
    \centering
    \includegraphics[width=\textwidth]{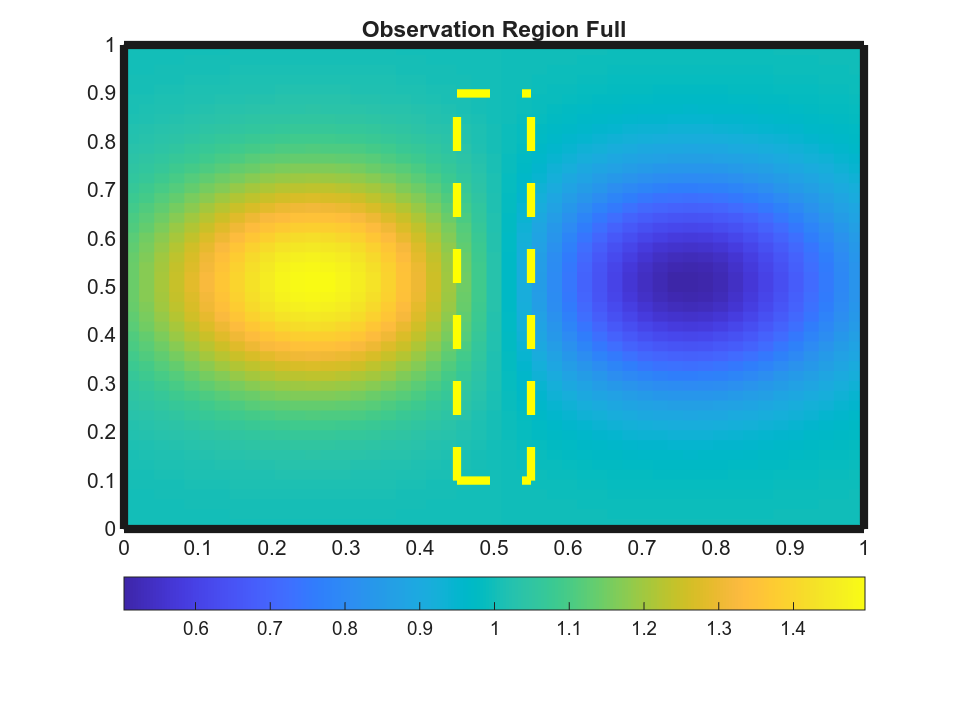}
    \footnotesize 
    Ground truth $a^\dagger$ \\
    Observation region: black\\
    Indenter: yellow dashed\\[1ex]
    \includegraphics[width=\textwidth]{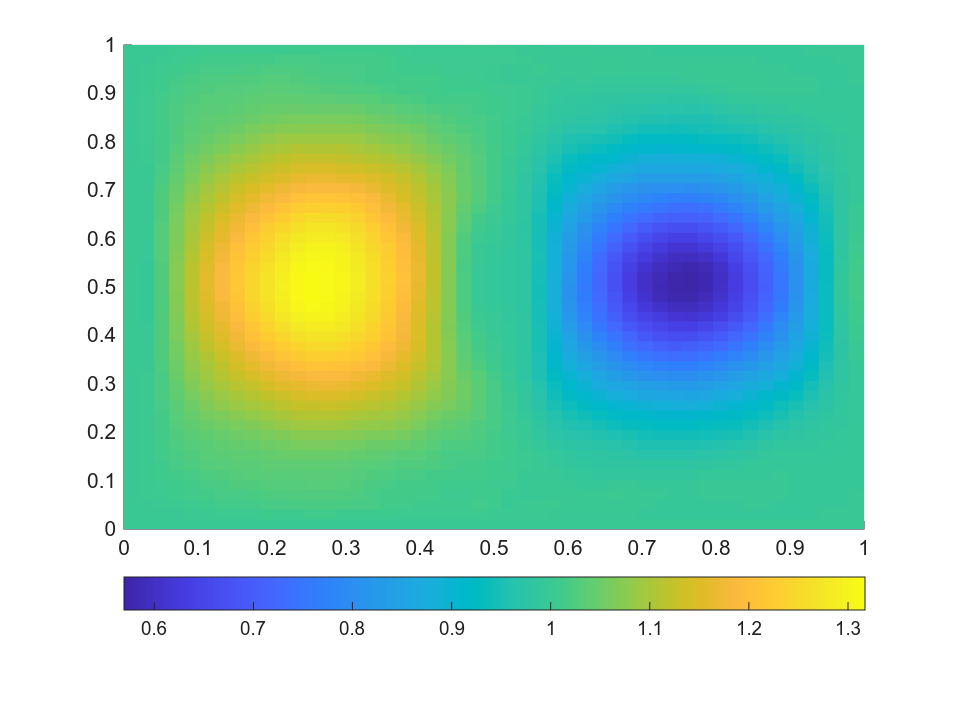}
    \footnotesize Reconstruction
\end{minipage}
\caption{Experiment~2: reconstruction from partial observations. Top row: ground truth $a^\dagger$ with indenter region displayed in dashed yellow. The respective observation regions are displayed as rectangles in black. Bottom: reconstruction ($\delta_{rel}= 0.1\%$.) (Left: observation on left half; Center: observation on right half; Right: full observation)}
\label{figex2}
\end{figure} 

\vspace{10pt}
\noindent
\textbf{Experiment 2.}
In our second experiment, we illustrate how the constraint can lead to a loss of information and identifiability. For this, we use the same discretization as before, and for the ground-truth $a^\dagger$ we use a sum of two Gaussians of opposite signs,~i.e., 
    \begin{equation*}
        a^\dagger(x_1,x_2) := 1 + \frac{1}{2} e^{-20(x_1-0.25)^2 +(x_2-0.5)^2} + \frac{1}{2} e^{-20(x_1-0.25)^2 +(x_2-0.5)^2} \,.
    \end{equation*} 
To test more general setups, we take a non-zero right-hand side $f = 6 (x_1(1-x_1) + x_2(1-x_2))$. In this experiment, we use data with random noise $\delta_{\text{rel}} = 0.1\%$, and terminate the algorithm after $300$ iterations. For the indenter, we select the characteristic function of the thin tall rectangle $[0.45,0.55] \times [0.1 , 0.9]$. Moreover, we consider three different observation setups:
    i) observation of $u$ on the left part $0 \leq x \leq 0.5$;
    ii) observation of $u$ on the right part $0.5 \leq x \leq 1$;
    iii) full observations. 
See Figure~\ref{figex2} (top row) for an illustration, where the ground-truth $a^\dagger$ is displayed together with the indenter shape (yellow dashed lines) and the respective observation region (black lines). 
The bottom row in Figure~\ref{figex2} displays the optimal results after at most 300 iterations. 

We see that with partial observations, the indenter acts as  a kind of barrier for the recovery of the coefficients. Only the part within the observation region can be reconstructed reliably. The result again illustrates the non-uniqueness of the problem, similar to the setup of Figure~\ref{fig:one}.

\begin{figure}[ht!]
\begin{minipage}{0.49\textwidth}
    \centering 
    \includegraphics[width=\textwidth]{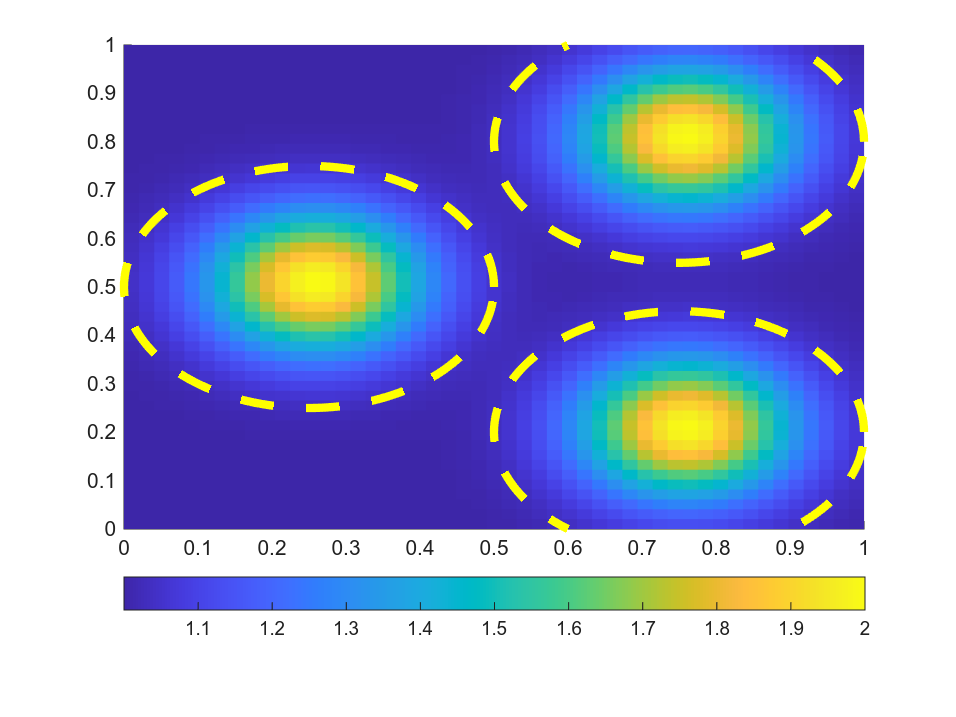}
    \footnotesize 
    Ground truth $a^\dagger$ \\
    Full observations \\
    Indenters: yellow dashed\\[1ex]
\end{minipage}
\begin{minipage}{0.49\textwidth}
    \centering 
    \includegraphics[width=\textwidth]{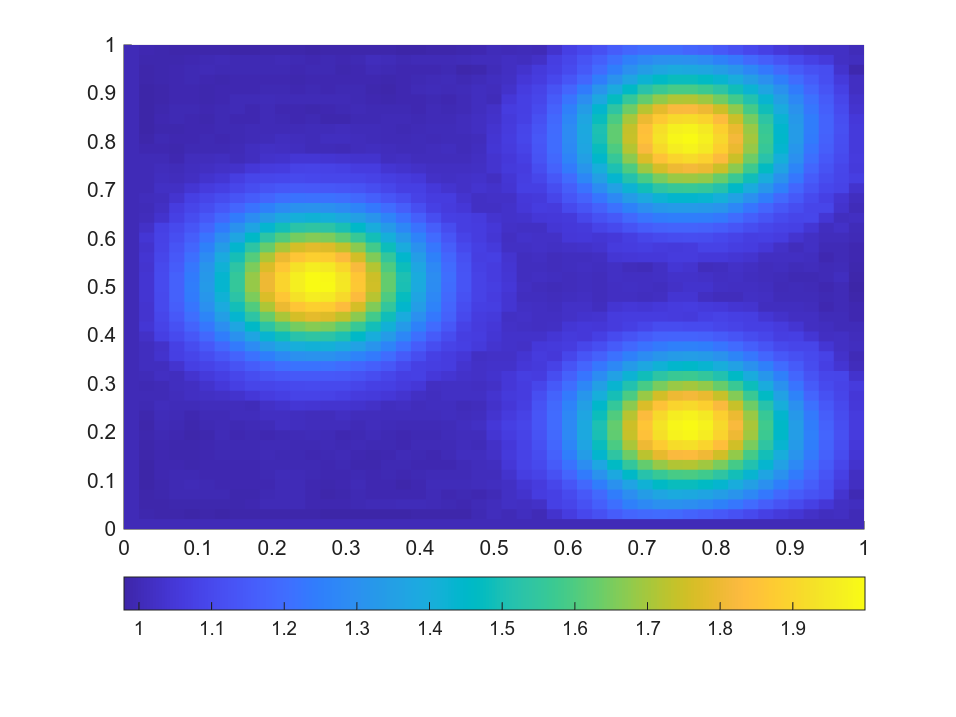}
    \footnotesize 
    Reconstruction using 3 \\
    independent measurements \\[1ex]
    \ 
\end{minipage}
\caption{Experiment 3: multiple measurements. Left: ground truth and shape of 3 (different) indenters. Right: resulting $a_k$  after $k = 1106$ iterations of method \eqref{Nest}.}
\label{figex3}
\end{figure}

\vspace{10pt}
\noindent
\textbf{Experiment 3.} In our third experiment, we illustrate the benefit of multiple observations in the inverse contact problem. With the same discretization as before, for the ground-truth $a^\dagger$ we use the sum of three Gaussians centered at three different points:
    \begin{equation*}
    \begin{split}
        a^\dagger(x_1,x_2) = 1 &+ 
        0.5 \, e^{-50 ( 
        (x_1-0.25)^2 + (x_2- 0.5)^2)}
        \\
        &+
        0.5 \, e^{-50 ( 
        (x_1-0.75)^2 + (x_2- 0.2)^2)} 
        +
        0.5 \, e^{-50 ( 
        (x_1-0.75)^2 + (x_2- 0.8)^2)} 
        \,,
    \end{split}
    \end{equation*}
Also, we choose the same non-zero right-hand side $f$ as Experiment~2. Furthermore, we consider three different, circularly shaped indenters, centered at the same locations as the three Gaussians in the ground-truth $a^\dagger$, and each with the same radius of $0.25$. In this way, each indenter essentially covers one of the Gaussians; see Figure~\ref{figex3} (left) for an illustration of the exact solution and the shape of the indenters (yellow dashed). 

It is clear that if the indenters are all applied simultaneously, none of the Gaussians of the ground truth can be recovered. However, this is not the case when they are used sequentially: We perform three experiments, where each time only one of the indenters is applied, providing three different solutions $u_1,u_2,u_3$. Our algorithm is designed to treat multiple data by default: the solutions $u_i$ are stored as columns of a matrix and the relevant operators ($B_i$, PDE-solvers, thresholding) operate columnwise, and the scalar products are implemented by employing the Frobenius inner product. In the single-data case, the data matrices collapse to a vector. Furthermore, we consider full observations, and thus the data consists of the set $\Kl{u_1, u_2, u_3}$, contaminated by noise. The noise is added in the same way as in Experiment~1 by using both random and discretization errors with a relative noise level of $0.1\%$. Figure~\ref{figex3} (right) illustrates the best iterate obtained within the first 1200 iterations of the Nesterov reconstruction scheme \eqref{Nest}. We can observe a reasonably accurate reconstruction, which is close to the ground truth. Further experiments not shown here confirm that if only one measurement of $u_i$ is used for reconstruction, then the part of $a^\dagger$ covered by the indenter cannot be reconstructed, while the other two Gaussians can be partially recovered.  

\begin{figure}[ht!]
\begin{minipage}{0.49\textwidth}
    \centering 
    \includegraphics[width=\textwidth]{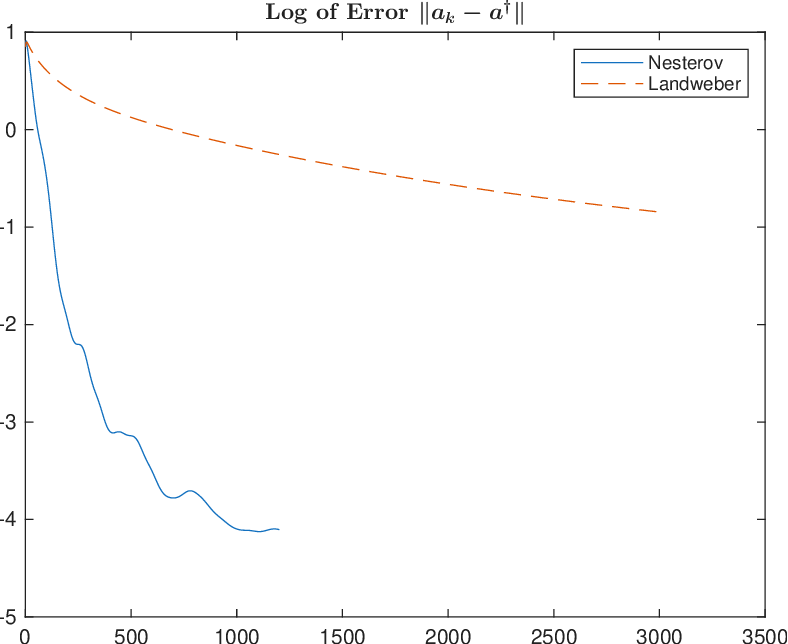}
    \footnotesize 
\end{minipage}
\begin{minipage}{0.49\textwidth}
    \centering 
    \includegraphics[width=\textwidth]{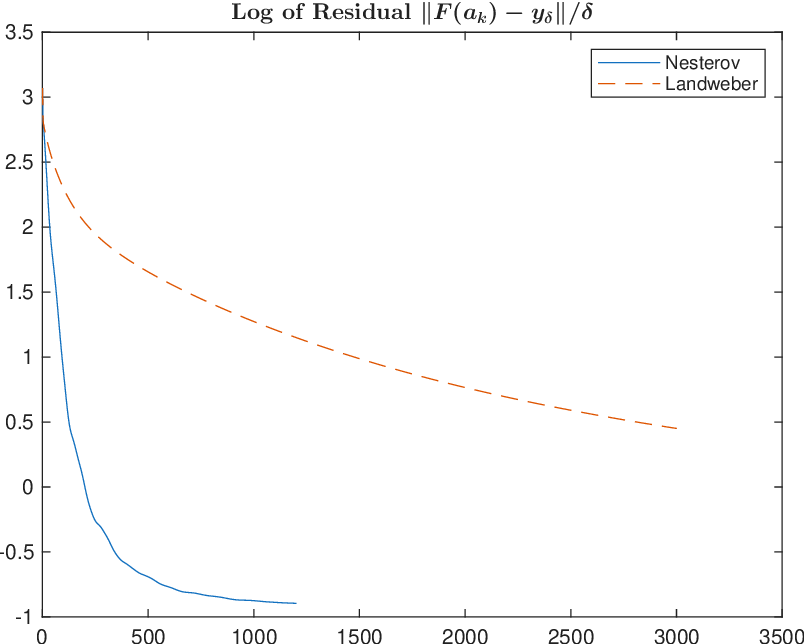}
    \footnotesize 
\end{minipage}
\caption{Experiment~3: convergence curves for the Nesterov-accelerated reconstruction method \eqref{Nest} (blue, full) and Landweber iteration \eqref{modland} (red, dashed) Left: logarithm of error $\norm{a_{k} - a^\dagger}_{H_0^1(\Omega)}$ vs.\ iterations. Right: logarithm of $\norm{F(a_k) -y^\delta}/\norm{y-y^\delta}$ vs iterations.}
\label{figex3plot}
\end{figure}

Finally, in Figure~\ref{figex3plot} (left) we display the convergence curves of the log of the error $\norm{a_{k} - a^\dagger}_{H_0^1}$ versus the iteration number $k$ for both the Landweber iteration \eqref{modland} (red) and the Nesterov-accelerated version \eqref{Nest} (blue). Furthermore, in Figure~\ref{figex3plot} (right) we display the log of the discrepancy factor $\norm{F(a_k) -y^\delta}/\norm{y-y^\delta}$ for both methods versus the iteration index $k$. The acceleration of the Nesterov scheme is clearly visible.

\section{Conclusion}\label{sect_conclusion}

In this paper, we considered the identification of a scalar coefficient in a PDE-based parameter estimation problem with contact constraints, serving as an idealized model of a membrane under forces and constrained by a barrier or indenter. We discussed both the forward and inverse parameter estimation problems, as well as uniqueness and non-uniqueness issues caused by the contact constraints. Furthermore, we considered the finite-dimensional versions of these problems, and discussed the design and implementation of corresponding solution and regularization approaches. Finally, we presented a number of numerical experiments demonstrating the performance of our proposed approaches, and illustrating both uniqueness and non-uniqueness situations.

\section{Acknowledgements \& Support}

This research was funded in part by the Austrian Science Fund (FWF) SFB 10.55776/F68 ``Tomography Across the Scales'', project F6805-N36 (Tomography in Astronomy) and project F6807-N36 (Tomography with Uncertainties), as well as 10.55776/P34981 ``New Inverse Problems of Super-Resolved Microscopy (NIPSUM)''. For open access purposes, the authors have applied a CC BY public copyright license to any author-accepted manuscript version arising from this submission. The financial support by the Austrian Federal Ministry for Digital and Economic Affairs, the National Foundation for Research, Technology and Development and the Christian Doppler Research Association is gratefully acknowledged. The authors thank Prof.~Christian Clason, University Graz, for fruitful discussions.

\bibliographystyle{plain}
{\footnotesize
\bibliography{extracted.bib}
}

\end{document}